\documentclass[11pt,a4paper,reqno]{amsart}
\usepackage{amsmath,amsfonts,amssymb,amsthm,amscd}
\usepackage[english]{babel}
\usepackage[mathscr]{eucal}
\usepackage{graphicx}
\usepackage{hyperref}

\renewcommand{\a}{\alpha}
\renewcommand{\b}{\beta}
\newcommand{\g}{\gamma}
\newcommand{\G}{\Gamma}
\renewcommand{\d}{\delta}
\newcommand{\D}{\Delta}
\newcommand{\h}{\chi}

\renewcommand{\l}{\lambda}

\newcommand{\m}{\mu}
\newcommand{\n}{\nu}
\newcommand{\om}{\omega}
\newcommand{\OM}{\Omega}
\renewcommand{\r}{\rho}
\newcommand{\s}{\sigma}

\renewcommand{\t}{\tau}
\renewcommand{\th}{\theta}
\newcommand{\e}{\varepsilon}
\newcommand{\f}{\varphi}
\newcommand{\F}{\Phi}
\newcommand{\x}{\xi}

\newcommand{\y}{\eta}

\newcommand{\p}{\psi}

\newcommand{\C}{{\mathbb C}}

\newcommand{\R}{{\mathbb R}}

\newcommand{\RR}{{\mathbb R}^2}

\newcommand{\T}{{\mathbf T}}

\newcommand{\Cc}{{\mathcal C}}

\newcommand{\Dc}{{\mathcal D}}

\newcommand{\Tc}{{\mathcal T}}
\newcommand{\Oc}{{\mathcal O}}

\newcommand{\curl}{{\rm curl}\,}

\newcommand{\diver}{{\rm div}\,}

\newcommand{\pd}{\partial}

\newcommand{\dt}{\partial_t}

\newcommand{\eps}{\varepsilon}

\newcommand{\supp}{\operatorname{supp\,}}

\newcommand{\loc}{\operatorname{{loc}}}

\newtheorem{theorem}{Theorem}[section]
\newtheorem{proposition}[theorem]{Proposition}
\newtheorem{lemma}[theorem]{Lemma}
\newtheorem{corollary}[theorem]{Corollary}

\theoremstyle{remark}
\newtheorem{remark}[theorem]{Remark}

\numberwithin{equation}{section}

\newcommand{\na}{\nabla}
\newcommand{\vti}{\tilde{v}}
\newcommand{\oti}{\tilde{\om}}

\begin{document}


\title[Uniqueness for Euler equations on singular domains]{Uniqueness for two dimensional incompressible ideal flow on singular domains}
\author{C. Lacave}

\address[C. Lacave]{Universit\'e Paris-Diderot (Paris 7)\\
Institut de Math\'ematiques de Jussieu - Paris Rive Gauche\\
UMR 7586 - CNRS\\
B\^atiment Sophie Germain \\
Case 7012\\
75205 PARIS Cedex 13\\
France.} \email{lacave@math.jussieu.fr}

\date{\today}

\begin{abstract}
The existence of a solution to the two dimensional incompressible Euler equations in singular domains was established in \cite{GV_lac}. The present work is about the uniqueness of such a solution when the domain is the exterior or the interior of a simply connected set with corners, although the velocity blows up near these corners. In the exterior of a curve with two end-points, it is showed in \cite{lac_euler} that this solution has some interesting properties, as to be seen as a special vortex sheet. Therefore, we prove the uniqueness, whereas the problem of general vortex sheets is open.
\end{abstract}

\maketitle

\tableofcontents

\section{Introduction}

The motion of a two dimensional flow can be described by the velocity 
$u(t,x) = (u_1,u_2)$ and the pressure $p$. Concerning incompressible ideal flow in an open set $\OM$, the pair $(u,p)$ verifies the Euler equations:
\begin{equation} \label{Euler}
\left\{
\begin{aligned}
 \pd_t u + u \cdot \na u + \na p & = 0, \quad  t > 0, x \in \Omega  \\
 \diver u & = 0, \quad  t > 0, x \in \Omega 
\end{aligned}
\right.
\end{equation}
endowed with an initial condition and an impermeability condition at the boundary $\pd \Omega$: 
\begin{equation} \label{conditions}
u\vert_{t=0} = u_0, \quad u \cdot \hat n\vert_{\pd \Omega} = 0.   
\end{equation}
The vorticity $\om$ defined by
\[ \om := \curl u = \pd_1 u_2 - \pd_2 u_1\]
plays a crucial role in the study of the ideal flow, thanks to the transport nature governing it:
\begin{equation}\label{transport}
\pd_t \om + u\cdot \na \om= 0 .
\end{equation}

When $\OM$ and $u_0$ are smooth, the well-posedness of system \eqref{Euler}-\eqref{conditions}  has been of course the matter of many works.  Starting from the paper of Wolibner in bounded domains \cite{Wo}, McGrath treated the case of the full plane \cite{MG}, and finally Kikuchi studied the exterior domains \cite{kiku}. In the case where the vorticity is only assumed to be bounded, existence and uniqueness of a weak solution has  been established by Yudovich in \cite{yudo}. We quote that the well-posedness result of Yudovitch applies to smooth  bounded domains, and to unbounded ones under further decay assumptions.

\medskip
We stress that all above studies require  $\pd \Omega$ to be at least  $C^{1,1}$. Roughly,  the reason is the following: due to the non-local character of the Euler equation, these works rely on   global in space estimates of $u$ in terms of  $\omega$. These estimates  {\em up to the boundary} involve  Biot and Savart type kernels, corresponding to operators such as $\na \Delta^{-1}$.  Unfortunately,  such operators are known  to behave badly in general non-smooth domains.  This explains why well-posedness results are dedicated to regular domains.

\medskip

However the case of a singular obstacle is physically relevant. For example, the study of the perturbation created by a plane wing stays a capital issue to determine the safety time between two landings in big airports.

Without solving the question of uniqueness, Taylor established in \cite{taylor} the existence of a global weak solution of \eqref{Euler}-\eqref{conditions} in a bounded sharp convex domain.  He used that $\Omega$ convex implies that  the solution $v$ of the  Dirichlet problem 
$$ \Delta v = f \: \mbox{ in } \: \Omega, \quad v\vert_{\pd \Omega} = 0 $$
belongs to $H^2(\Omega)$ when the  source term $f$ belongs to $L^2(\Omega)$, irrespective of the domain regularity.  Nevertheless, this interesting result still leaves aside many situations of practical interest, notably flows around irregular obstacles. Recently, the article \cite{lac_euler} gave such a result in the exterior of a $C^2$ Jordan arc, where it is noted that the velocity blows up near the end-points of the arc. In particular, it shows  that the previous property on the Dirichlet problem is false in domains with some bad corners.

The question of the existence of global weak solutions is now solved for a large class of singular domains in \cite{GV_lac}. The authors therein considered two kinds of domains: any open bounded domain where we retrieve a fixed (possibly zero) number of closed sets with positive capacity, and any exterior domain of one connected closed set with positive capacity. 

{\em Our goal here is to prove that such a solution is unique if the domain is bounded, simply connected with some corners, or if it is the complementary of a closed simply connected bounded set with some corners. We prove the uniqueness for an initial vorticity which is bounded, compactly supported in $\OM$ and having a definite sign}.

\medskip

More precisely, we consider two kinds of domains. On one hand, we denote by $\Omega$ a bounded, simply connected open set, such that $\pd \OM$ has a finite number of corners $z_i$ with angles $\a_i$ (i.e. locally, $\OM$ coincides with the sector $\{z_i+(r\cos \th, r\sin \th); r>0, \th_i<\th<\th_i+\a_i\}$). On the other hand, we denote by $\OM:=\R^2\setminus \Cc$, where $\Cc$ is a bounded, simply connected closed set, such that $\pd \OM$ has a finite number of corners.

\medskip
To define a global  weak solution to the Euler equation,  let us point out that the  space $L^2(\Omega)$ is not suitable for  weak solutions in unbounded domain.  Working with square integrable velocities in exterior domains is too restrictive (see page \pageref{1/x} to note that $u$ behaves in general like $1/|x|$ at infinity), so we consider initial data satisfying 
\begin{equation} \label{typeinitialdata}
u_0 \in L^2_{\loc}(\overline{\Omega}), \quad u_0 \rightarrow 0 \:  \mbox{ as } \: |x| \rightarrow +\infty,  \quad  \curl u_0 \in L^\infty_c (\Omega), \quad \diver u_0 = 0, \quad u_0 \cdot \hat n\vert_{\pd \Omega} = 0. 
\end{equation}
Note that the divergence free condition and this last impermeability condition have to be understood in the weak sense: \begin{equation}\label{imperm}
 \int_\Omega u^0 \cdot h = 0  \quad \text{for all } h \in G_{c}(\Omega):=\{w\in L^2_{c}(\overline{\Omega}) \ : \ w=\nabla p, \ \text{ for some } p\in H^1_{\loc}(\Omega)\}.
 \end{equation}
Let us stress that this set of initial data is large: we will show later that for any function $\om_0 \in L^\infty_c(\Omega)$, there exists $u_0$ verifying \eqref{typeinitialdata} with $\curl u_0= \om_0$.

Similarly, the weak form of the  divergence free and tangency conditions on the Euler solution $u$  will read:  
\begin{equation} \label{imperm2}
\forall \varphi \in \mathcal{D}\left([0,+\infty); G_c(\Omega)\right), \quad \int_{\R^+} \int_\Omega u  \cdot \na \varphi = 0.
\end{equation}
Finally, the weak form  of the momentum equation on $u$ will read:
\begin{equation} \label{Eulerweak}
\forall \, \varphi \in \mathcal{D}\left([0, +\infty[ \times \Omega\right) \text{ with } \diver \varphi = 0, \quad  \int_0^{\infty} \int_\Omega \left( u \cdot \pd_t \varphi +   (u \otimes u) : \na \varphi \right)  = -\int_\Omega u_0 \cdot \varphi(0, \cdot) .
\end{equation}

For $\OM$ an open bounded simply connected domain, or $\OM$ the complementary of a compact simply connected domain $\mathcal{C}$, we get the existence of a weak solution from \cite{GV_lac}:
\begin{theorem} \label{theorem1}
Assume that $u_0$ verifies \eqref{typeinitialdata}. Then there exists 
$$u \in L^\infty_{\loc}(\R^+; L^2_{\loc} (\overline{\Omega})), \:  \curl u \in   L^\infty(\R^+;L^1 \cap L^\infty(\Omega)),$$
which is a global weak solution of \eqref{Euler}-\eqref{conditions} in the sense of \eqref{imperm2} and \eqref{Eulerweak}.
\end{theorem}
In a few words, this existence result  follows from a compactness argument, performed on a sequence of   solutions $u_n$  of the Euler equations on  the sequence of approximating domains $\Omega_n$. A key ingredient of the proof is the so-called $\Gamma$-convergence of $\Omega_n$ to $\Omega$ (see \cite{GV_lac} for the details).

\medskip
 
The main result of this article concerns the uniqueness of global weak solutions, when the initial vorticity has definite sign.
 \begin{theorem} \label{main 1}
Let $\Omega$ be a bounded, simply connected open set, such that $\pd \OM$ has a finite number of corners with angles greater than $\pi/2$ and let $u_0$ verifying \eqref{typeinitialdata}. If $\curl u_0$ is non-positive (respectively non-negative), then there exists a unique global weak solution of the Euler equations on $\OM$ verifying
$$u \in L^\infty_{\loc}(\R^+; L^2_{\loc} (\Omega)), \:  \curl u \in   L^\infty(\R^+;L^1 \cap L^\infty(\Omega)).$$
 \end{theorem}
 
 In exterior domains, the vorticity is not sufficient to uniquely determine the velocity. We need the circulation around $\Cc$. As we will see in Subsection \ref{sect : exist}, for $u_0$ verifying \eqref{typeinitialdata}, we can define the initial circulation:
 \[ \g_0:= \oint_{\pd \Cc} u_0\cdot \hat \tau\, ds.\]
Inversely, let us mention that we can fix independently the vorticity and the circulation: we will show that for any function $\om_0 \in L^\infty_c(\Omega)$ and any real number $\g\in \R$, there exists a unique $u_0$ verifying \eqref{typeinitialdata} with $\curl u_0= \om_0$ and with circulation around $\mathcal{C}$ equal to $\g$.
 
 Assuming a sign condition on $\g_0$, we will prove a uniqueness theorem in exterior domains.
 \begin{theorem} \label{main 2}
Let $\OM:=\R^2\setminus \Cc$, where $\Cc$ is a compact, simply connected set, such that $\pd \OM$ has a finite number of corners with angles greater than $\pi/2$. Let $u_0$ verifying \eqref{typeinitialdata}. If $\curl u_0$ is non-positive and $\g_0\geq -\int \curl u_0$(respectively $\curl u_0$ non-negative and $\g_0 \leq -\int \curl u_0$), then there exists a unique global weak solution of the Euler equations on $\OM$, verifying
$$u \in L^\infty_{\loc}(\R^+; L^2_{\loc} (\Omega)), \:  \curl u \in   L^\infty(\R^+;L^1 \cap L^\infty(\Omega)).$$
 \end{theorem}
 
In particular, we will also prove that the velocity blows up near the obtuse corners: if $\pd \OM$ admits at $z_0$ a corner of angle $\a$, then the velocity behaves near $z_0$ like $\dfrac{1}{|x-z_0|^{1-\frac{\pi}{\a}}}$. We refind that in the case where $\mathcal{C}$ is a Jordan arc (see \cite{lac_euler}) the velocity blows up like the inverse of the square root of the distance near the end-points ($\a=2\pi$). 

Although it is possible to show the existence of weak solution for $\om_0\in L^1\cap L^p(\OM)$, with some $p>1$ (see \cite{DiPernaMajda} in smooth domains and \cite{GV_lac} in non-smooth domains), we recall that the result of uniqueness requires $p=\infty$. Indeed, in the proof of Yudovich, we use that in smooth domains $\om_0\in L^1\cap L^\infty(\OM)$ implies that the velocity belongs to $L^\infty \cap W^{1,p}(\OM)$ for all $p\in (1,\infty)$. As far as we know, the only result with lower regularity is the case of a vorticity $\om_0=\tilde \om_0 + \a \d_{z_0}$ where $\tilde \om_0\in L^1\cap L^\infty$ is constant near the Dirac mass (see \cite{mar_pul,lac_miot}). However, in that case, the authors use the explicit form of the singularity. Therefore, we manage here to adapt the Yudovich proof for a velocity  not bounded up to the boundary. For example, in a  bounded domain with a cusp point ($\a=2\pi$), the velocity belongs in $L^p\cap W^{1,q}(\OM)$ only for $p<4$ and $q<4/3$, and proving the uniqueness could seem surprising when we keep in mind that we need $W^{1,p}$ for all $p$ large in the Yudovich proof. The key here is to use the sign condition in Theorems \ref{main 1} and \ref{main 2} in order to prove that the vorticity never meets the boundary.

For a sake of clarity, we assume that $\pd \OM$ is locally a corner, but we can replace a corner by a singular point, where the jump of the tangent angle is equal to $\a$ (see Remark \ref{DT loc} for a discussion concerning the optimal domain regularity).

\medskip

The remainder of this work is organized in six sections. We introduce in Section \ref{sect : 2} the biholomorphism $\Tc$ and the Biot-Savart law (law giving the velocity in terms of the vorticity) in the interior or the exterior of one simply connected domain. We will recall the existence of weak solution in this section, and derive some formulations (on vorticity and on extensions in $\R^2$).  We will take advantage of this section to show that the weak solution is a renormalized solution in the sense of DiPerna-Lions \cite{dip-li}, which will allow us to prove that the $L^p$ norm of the vorticity for $p\in [1,\infty]$, the total mass $\int_{\OM} \om(t,\cdot)$ and the circulation of the velocity around $\Cc$ are conserved quantities.

Let us mention that the explicit form of the Biot-Savart law is one of the key of this work, and it explains why $\OM$ is assumed to be the interior or the exterior of a simply connected domain. This law will read
\[ u(t,x) = D\Tc (x)^T R[\om]\]
where $R[\om]$ is an integral operator. Using classical elliptic theory, we will obtain the exact behavior of the biholomorphism $\Tc$ near the corners, and then the behavior of the velocity. We note that the blow-up is stronger if the angle $\a$ is bigger. Unfortunately, the following study needs sometimes that the integral operator $R[\om]$ verifies good estimates, which are possible only if we assume that all the angles $\a_i$ are greater than $\pi/2$ (namely in Proposition \ref{biot est} to prove that $R[\om]$ is bounded,  in Lemma \ref{ortho} to establish the equation verified by the extended functions, in Lemma \ref{W11} to use the renormalization theory).

Section \ref{sect : 3} is the central part of this paper: we will prove that the support of $\om$ never meets the boundary if we assume that the characteristics corresponding to \eqref{transport} exist and are differentiable. The idea is to introduce a good Liapounov function, which blows up if the trajectories meet the boundary. Next, we will establish some estimates implying that this Liapounov energy is bounded which will give the result. Although we cannot say that the characteristics are regular for weak solutions, this computation gives us an excellent intuition.

In light of this proof, we rigorously prove in Section \ref{sect : 4}, thanks to the renormalization theory, that we have the same property, even if we do not consider the characteristics.

Finally, we prove Theorems \ref{main 1}-\ref{main 2} in Section \ref{sect : 5}. We will introduce  $v:= K_{\R^2} * \om$, where $K_{\R^2}$ is the Biot-Savart kernel in the full plane. As $\om$ does not meet the boundary, it means that $\diver v = \curl v \equiv 0$ in a neighborhood of the boundary, i.e. $v$ is harmonic therein. This provides in particular a control of its $L^\infty$ norm (as well as the $L^\infty$ norm for the gradient) by its $L^2$ norm. Although the total velocity is not bounded near the boundary, but just integrable, this argument allows us to yield a Gronwall-type estimate, as Yudovich did.

\medskip

Therefore, the fact that the support of the vorticity stays far from the boundary will imply the uniqueness result. This idea was already used in \cite{lac_miot}, in the case of one Dirac mass in the vorticity. In this article, we consider the Euler equations in $\R^2$ when the initial vorticity is composed by a regular part $L^\infty_c$ and a Dirac mass. The equation is called the system mixed Euler/point vortex and derived in \cite{mar_pul}. When trajectories exist, it is proved that they do not meet the point vortex in \cite{mar_pul} if the point vortex moves on the influence of the regular part, and in \cite{mar} if the Dirac is fixed. The method is also constructed on Liapounov functions. An important issue in \cite{lac_miot} is to generalize this result when trajectories are not regular. The Lagrangian formulation gives us a helpful intuition, it is the reason why we choose first to present the proof of uniqueness assuming the differentiability of trajectories (Section \ref{sect : 3}). Moreover, proving in Section \ref{sect : 4} that the vorticity never meets the boundary, we state that the ``weak'' Lagrangian flow coming from the renormalization theory evolves in the area far from the corners. As the velocity is regular enough in this region, we can conclude that the flow is actually classical and regular.

Section \ref{sect : technical} is devoted to the proofs of some technical lemmas.

We finish this article by Section \ref{sect : 6} with some final comments. 
In the exterior of the Jordan arc (see \cite{lac_euler}), we will make a parallel with the vortex sheet problem. We will also give some explanations about the  sign assumptions in the main theorems.

\medskip

We warn the reader that we write in general the proofs in the case of exterior domains. In this kind of domain, we have to take care of integrability at infinity, to control the size of the support of the vorticity, and we have also to consider harmonic vector fields and circulations of velocities around $\Cc$. The proofs in the case of bounded domains are strictly easier, without additional arguments. We will make sometimes some remarks about that.

\section{Biot-Savart law and existence}\label{sect : 2}

As in \cite{ift_lop_euler, lac_euler, lac_small}, the crucial assumption is that we work in dimension two outside (or inside) one simply connected domain. Identifying $\R^2$ with the complex plane $\C$, there exists a biholomorphism $\Tc$ mapping $\OM$ to the exterior (resp. to the interior) of the unit disk. Thanks to this biholomorphism, we will obtain an explicit formula for the Biot-Savart law: the law giving the velocity in terms of the vorticity. This explicit formula will be used to construct the Liapounov function. We give in the following subsection the properties of this Riemann mapping.

\subsection{Conformal mapping}\

Let $\OM$ as in Theorem \ref{main 2} (resp. as in Theorem \ref{main 1}), then the Riemann mapping theorem states that there exists a unique biholomorphism $\Tc$ mapping $\OM$ to $\overline{B(0,1)}^c$ (resp. to $B(0,1)$) such that $\Tc(\infty)=\infty$ and $\Tc'(\infty)\in \R^+_*$ (resp. $\Tc(z_0)=0$ and $\Tc'(z_0)\in \R^+_*$, for a $z_0\in \OM$). We remind that the last two conditions mean  
$$\mathcal{T}(z) \sim \lambda z, \quad   |z| \sim +\infty, \quad \mbox{ for some } \:  \lambda> 0. $$

\begin{theorem} \label{grisvard}
Let assume that $\pd \OM$ is a $C^\infty$ Jordan curve, except in a finite number of point $z_1$, $z_2$, ..., $z_n$ where $\pd \OM$ admits corner of angle $\a_i>\frac{\pi}{2}$ (i.e. $\OM$ coincides locally with the sector $\{z_i+(r\cos \th, r\sin \th); r>0, \th_i<\th<\th_i+\a_i\}$). Then the biholomorphism $\Tc$ defined above satisfies
 \begin{itemize}
\item $\Tc^{-1}$ and $\Tc$ extend continuously up to the boundary;
\item $D \Tc^{-1}$ extends continuously up to the boundary, except at the points $\Tc(z_i)$ with $\a_i<\pi$ where $D \Tc^{-1}$ behaves like $1/|y-\Tc(z_i)|^{1-\dfrac{\a_i}{\pi}}$;
\item $D \Tc$ extends continuously up to the boundary, except at the points $z_i$ with $\a_i>\pi$ where $D \Tc$ behaves like $1/|x-z_i|^{1-\dfrac{\pi}{\a_i}}$;
\item $D^2 \Tc$ belongs to  $L^p_{\loc}(\overline{\OM})$ for any $p<4/3$.
\end{itemize}
\end{theorem}

\begin{proof} As $\pd \OM$ is $C^{0,\a}$, the Kellogg-Warschawski theorem (Theorem 3.6 in \cite{pomm-2}) states directly that $\Tc$ and $\Tc^{-1}$ is continuous up to the boundary. For the behavior of the derivatives, we use the classical elliptic theory: let
\[ u(x):= \ln | \Tc(x) |.\]
As $\Tc$ is holomorphic, we have that
\[ \D u = 0 \text{ in } \OM \text{ and } u=0 \text{ on } \pd \OM.\]
To localize near each corners, we can introduce a smooth cutoff function $\chi$ supported in a small neighborhood of $z_i$. Therefore, we are exactly in the setting of elliptic studies:
\begin{equation}\label{elliptic}
 \D (u\chi) = f \in C^\infty \text{ in } O_{i} \text{ and } u=0 \text{ on } \pd O_{i},
\end{equation}
where $O_{i}$ is the sector $\{z_i+(r\cos \th, r\sin \th); r>0, \th_i<\th<\th_i+\a_i\}$. The standard idea is to compose by $z^{\pi/\a_i}$ in order to maps the sector on the half plane, where the solution of the elliptic problem $g$ is smooth. Therefore, we have that
\[ u\chi = g \circ z^{\pi/\a_i},\]
which implies that
\begin{equation}\label{approx}
 \na u \approx r^{\pi/\a_i-1} \text{ and } \ \na^{2} u \approx r^{\pi/\a_i-2}. 
\end{equation}

More precisely, we used the so-called shift theorem in non-smooth domain (see the preface of \cite{grisvard}): there exist numbers $c_k$ such that
\[u\chi-\sum c_k v_k \in W^{m+2,p}(O_i\cap B(0,R)), \ \forall R>0 \]
where the $k$ in the summation ranges over all integers such that
\[ \pi/\a_i \leq k\pi/\a_i <m+2-2/p\]
and with
\begin{itemize}
\item $v_k = r^{k\pi/\a_i} \sin(k\pi \th/\a_i)$ if $k\pi/\a_i$ is not an integer;
\item $v_k = r^{k\pi/\a_i} [ \ln r \sin(k\pi \th/\a_i) + \th \cos(k\pi \th/\a_i)]$ if $k\pi/\a_i$ is an integer.
\end{itemize}
In this theorem, $r$ denotes the distance between $x$ and $z_i$: $r:=|x-z_i|$.

We apply it for $m=1$ and $p=2$. As $H^3_{\loc}(\R^2)$ embeds in $C^0$, we see again that $u$ is continuous up to the boundary. 

If $\pi< \a_i \leq 2\pi$ then $1/2 \leq\pi/\a_i < 1$, which gives that $\pi/\a_i$ cannot be an integer. Then, the shift theorem states that $D(u\chi) -\sum c_k Dv_k$ belongs to $H^2_{\loc}(\overline{O_i})$, so it belongs to $C^0$. Thanks to formula of $v_k$, we see that $Du$ is continuous up to the boundary, except near $z_i$ where $Du =  \Oc(r^{\pi/\a_i-1})$. Next, we derive once more to obtain that $D^2 (u\chi) -\sum c_k D^2 v_k$ belongs to $H^1_{\loc}(\overline{O_i})$, so it belongs to $L^p_{\loc}(\overline{O_i})$ for any $p$. As $\sum c_k D^2v_k = \Oc (r^{\pi/\a_i-2})$, with $2-\pi/\a_i<3/2$, then $D^2 u$ belongs to $L^p_{\loc}(\overline{O_i})$ for any $p<4/3$.

The case $\a_i = \pi$ is not interesting because we assume that $z_i$ is a singular point.

If $\pi/2 < \a_i <\pi$, then we note that $\pi/\a_i$ is not an integer and that $k\pi/\a_i <2$ is obtained only for $k=1$. We apply the above argument to see that $u$ and $Du$ is continuous up to the boundary, and $D^2 u$ belongs to $L^p_{\loc}(\overline{O_i})$ for any $p<2$.

Therefore, the shift theorem establishes rigorously that $u = \Oc (r^{\pi/\a_i})$ and $Du = \Oc (r^{\pi/\a_i-1})$ if all the angles are greater than $\pi/2$.
We show now that $Du$ and $D\Tc$ have the same behavior. 

On one hand, differentiating $u$, we have
\[\na u(x)= \frac{\Tc(x)}{|\Tc(x)|^2}D\Tc(x)\]
hence
\begin{equation}\label{ut}
 | \na u(x) |_\infty \leq 4 | D\Tc(x) |_\infty 
\end{equation}
where $| A |_\infty = \max |a_{ij}|$. Indeed, by continuity of $\Tc$, we have that $|\Tc(x)|=\sqrt{\Tc_1^2(x)+\Tc_2(x)^2}\geq 1/2$ near the boundary.

On the other hand,
\[ \frac{\Tc(x)}{|\Tc(x)|^2} = \na u(x) D\Tc(x)^{-1}.\]
By continuity of $\Tc$, there exists a neighborhood of $\pd \OM$ such that $|\Tc(x)|\leq 2$. Then, near the boundary, we have
\[\frac12\leq \frac{1}{|\Tc(x)|}\leq 2 \sqrt{2} |\na u(x)|_\infty |D\Tc(x)^{-1}|_\infty.\]
Moreover, as $\Tc$ is holomorphic, $D\Tc$ is a matrix $2\times 2$ on the form $\begin{pmatrix} a&b\\-b&a \end{pmatrix}$. We deduce from this form that $D\Tc(x)^{-1}=\frac{1}{\det D\Tc(x)} D\Tc(x)^T$. We use that $\det D\Tc(x) = a^2 + b^2 \geq | D\Tc(x) |_\infty^2$ to get
\begin{equation}\label{tu}
 | D\Tc(x) |_\infty \leq 4\sqrt{2} |\na u(x)|_\infty.
\end{equation}
Putting together \eqref{approx}, \eqref{ut} and \eqref{tu}, we can conclude on the behavior of $D\Tc$.

Differentiating once more, we obtain the result for $D^2 \Tc$.

Finally, as  $u = \Oc (r^{\pi/\a_i})$, we state that
\[ | \Tc(x)| = 1 + \Oc (r^{\pi/\a_i}), \  \Tc(x) = \Tc(z_i) + \Oc (|x-z_i|^{\pi/\a_i}), \ \Tc^{-1}(y) = z_i + \Oc (|y- \Tc(z_i)|^{\a_i/\pi}).\]
Next, we use the fact that $D\Tc(x) =  \Oc (|x-z_i|^{\pi/\a_i-1})$ to write
\[ D\Tc^{-1}(y)= \Bigl(D\Tc(\Tc^{-1}(y)) \Bigl)^{-1} = \Oc \Bigl(\frac{1}{(|y- \Tc(z_i)|^{\a_i/\pi})^{\pi/\a_i-1}}\Bigl)= \Oc (|y- \Tc(z_i)|^{\a_i/\pi-1})\]
which ends the proof.
\end{proof}

We refind the result of the exterior of the curve (see \cite{lac_euler}): $\a=2\pi$ gives that $D\Tc$ behaves like $1/\sqrt{|x|}$. In that paper, we found the behavior of $D\Tc$ thanks to the explicit formula of $\Tc$. The Joukowski function $G(z)=\frac12(z+\frac1z)$ maps the exterior of the unit disk to the exterior of the segment $[(-1,0),(1,0)]$. Then, in the case of this segment $\Tc=G^{-1}$ and we can compute that
\[ D\Tc(z) = z \pm \frac{z}{\sqrt{z^2-1}}.\]

We also note that $D\Tc$ near a corner ($\a>\pi$) is less singular than around a cusp (as the intuition).

\begin{remark}\label{DT loc}
This kind of theorem will be useful to remark that the velocity in the exterior of a square blows-up like $1/|x|^{1/3}$ near the corner. However, the only things that we need in the sequel are:
\begin{itemize}
\item  there exists $p_0>2$ such that $\det D\Tc^{-1}$ belongs to $L^{p_0}_{\loc}(\overline{\OM})$: property holding true if all the corners $z_i$ have angles $\a_i$ greater than $\pi/2$ (as in Theorems \ref{main 1}-\ref{main 2});
\item $D\Tc$ belongs to  $L^p_{\loc}(\overline{\OM})$ for any $p<4$ and $D^2 \Tc$ belongs to  $L^p_{\loc}(\overline{\OM})$ for any $p<4/3$.
\end{itemize}
Therefore, Theorems \ref{main 1}-\ref{main 2} can be applied for any  simply connected domain (or exterior of a simply connected set) such that the two previous points hold true. For a sake of clarity, we express the theorems when the boundary is locally a corner at $z_i$, but we can generalize for $\OM$ such that $\pd \OM$ is a $C^{1,1}$ Jordan curve except at a finite number of points $z_i$. In these points, we would define 
\[\a_i := \lim_{s\to 0} \arg (\G'(s_i+s), \G'(s_i-s)) + \pi,\]
where $\G$ is a parametrization of $\pd \OM$ and $z_i=\G(s_i)$. Indeed, up to a smooth change of variable, the Laplace equation in $\OM$ \eqref{elliptic} turns into a divergence form elliptic equation in the exterior of a corner, and we would use results related to elliptic equations in polygons, see \cite{Mazya}.

Actually, we treat any domains $\OM$ which is the interior or the exterior of a simply connected set, such that the solution of 
\[ \D \p = f, \quad \p\vert_{\pd \OM}=0\]
belongs at least to $W^{2,1}$ for $f$ smooth.
Thanks to the works near a corner, we know that it is the case if $\pd \OM$ is smooth except in a finite number of points. A natural question is ``can we assume that the boundary is lipschitz ?''. The answer is no: in \cite{kenig}, the authors study the regularity that we can expect if we just assume that the boundary is lipshitz, and they cannot obtain estimates in $W^{2,1}$. Actually, they construct a counter-exemple: there exists $\OM$ a $C^1$ domain and $f\in C^{\infty}$ for which the second derivative of the solution $\p$ does not belong to $L^1(\OM)$. Then, even for $C^1$, the elliptic properties stated in the previous theorem do not hold. Assuming that the boundary have a finite number of corners, and smooth elsewhere, can be interpreted as assuming that the set of singular points is negligible.
\end{remark}

\begin{remark}
In the previous proof, we have chosen a method based on the shift theorem, which gives an elegant proof for large angle. For an alternative proof, which also holds for small angle, the reader can read \cite[Prop. 2.1]{lac_miot_wang}, where we only use the Kellogg-Warschawski theorem.
\end{remark}

\bigskip

The previous theorem is about the behavior near the obstacle. In the case of an unbounded domain (as in Theorem \ref{main 2}), we will need the following proposition about the behavior of $\Tc$ at infinity.

\begin{proposition} \label{T-inf}
If $\Tc$ is a biholomorphism from $\OM$ to the exterior of the unit disk such that $\Tc(\infty) = \infty$ and $\Tc'(\infty) \in \R^*_+$, then there exist $(\b,\tilde\b)\in \R^+_*\times \C$ and a holomorphic function $h:\OM \to \C$ such that
\[ \Tc(z) = \b z+ \tilde \b + h(z)\]
with
\[ h(z) = \Oc\Bigl(\frac1{|z|}\Bigl) \text{ and } h'(z)=\Oc \Bigl(\frac1{|z|^2}\Bigl),  \text{ as } |z|\to \infty.\]
Moreover, $\Tc^{-1}$ admits a similar development.
\end{proposition} 

\begin{proof} We consider $E:= \Tc^{-1}(B(0,2)\setminus B(0,1)) \cup \Cc$, which is an open, bounded, connected, simply connected and smooth subset of the plane. Then, the map $H := \Tc/2$ is a biholomorphism between $E^c$ and $B(0,1)^c$, and we can apply Remark 2.5  of \cite{lac_euler} to end this proof.
\end{proof}

\subsection{Biot-Savart Law} \label{sect : biot}\

One of the keys of the study for two dimensional ideal flow is to work with the vorticity equation, which is a transport equation. For example, in the case of a smooth obstacle, if we have initially $\om_0:=\curl u_0 \in L^1\cap L^\infty$, then $\| \om(t,\cdot)\|_{L^p}= \| \om_0\|_{L^p}$ for all $t,p$.  So, we have some estimates  for the vorticity, and the goal is to establish estimates for the velocity. For that, we introduce the Biot-Savart law, which gives the velocity in terms of the vorticity. Another advantage of the two dimensional space is that we have explicit formula in the exterior of one obstacle, thanks to complex analysis and the identification of $\R^2$ and $\C$.

Let $\OM$ be the exterior (resp. the interior) of a bounded, closed, connected, simply connected subset of the plane, the boundary of which is a Jordan curve. Let $\Tc$ be a biholomorphism from $\OM$ to $(\overline B(0,1))^c$ (resp. $B(0,1)$) such that $\Tc(\infty)=\infty$ (resp. $\Tc(z_0)=0$).

We denote by $G_{\OM}=G_{\OM} (x,y)$ the Green's function, whose the formula is:
\begin{equation*}
G_{\OM}(x,y)=\frac{1}{2\pi}\ln \frac{|\Tc(x)-\Tc(y)|}{|\Tc(x)-\Tc(y)^*||\Tc(y)|}
\end{equation*}
writing $x^*=\frac{x}{|x|^2}$. The Green's function verifies: 
\begin{equation*}
\D_y G_{\OM}(x,y)=\d(y-x) \ \forall x,y\in \OM, \quad G_{\OM}(x,y)=0  \ \forall (x,y)\in \OM \times \pd \OM, \quad G_{\OM}(x,y)=G_{\OM}(y,x)\ \forall x,y\in \OM.
\end{equation*}

The kernel of the Biot-Savart law is $K_{\OM}=K_{\OM}(x,y) := \na_x^\perp G_{\OM}(x,y)$. With $(x_1,x_2)^\perp=\begin{pmatrix} -x_2 \\ x_1\end{pmatrix}$, the explicit formula of $K_{\OM}$ is given by 
\begin{equation*}
K_{\OM}(x,y)=\dfrac{1}{2\pi} D\Tc^T(x)\Bigl(\dfrac{(\Tc(x)-\Tc(y))^\perp}{|\Tc(x)-\Tc(y)|^2}-\dfrac{(\Tc(x)- \Tc(y)^*)^\perp}{|\Tc(x)- \Tc(y)^*|^2}\Bigl)
\end{equation*}
and we introduce the notation
\[K_{\OM}[f]=K_{\OM}[f](x):=\int_{\OM} K_{\OM}(x,y)f(y)dy,\]
with $f\in C_c^\infty({\OM})$.

In unbounded domain, we require information on far-field behavior of $K_{\OM}$. We will use several times the following general relation:
\begin{equation}
\label{frac}
\Bigl| \frac{a}{|a|^2}-\frac{b}{|b|^2}\Bigl|=\frac{|a-b|}{|a||b|},
\end{equation}
which can be easily checked by squaring both sides. Using the behavior of $D\Tc$ at infinity (Proposition \ref{T-inf}), we obtain for large $|x|$ that 
\begin{equation*}
 |K_{\OM}[f]|(x)\leq \dfrac{C_2}{|x|^2},
 \end{equation*}
 where $C_2$ depends on the size of the support of $f$. 

The vector field $u=K_{\OM}[f]$ is a solution of the elliptic system:
\begin{equation*}
\diver u =0 \text{ in } {\OM}, \quad \curl u =f \text{ in } {\OM}, \quad u\cdot \hat{n}=0 \text{ on } \pd \OM, \quad \lim_{|x|\to\infty}|u|=0.
\end{equation*}

If we consider a non-simply connected domain (as in Theorem \ref{main 2}), the previous system has several solutions. To uniquely determine the solution, we have to take into account the circulation.
Let $\hat{n}$ be the unit normal exterior to $\OM$. In what follows all contour integrals are taken in the counter-clockwise sense, so that $\oint_{\pd \Cc} F\cdot \hat \t \, ds=-\oint_{\pd \Cc} F\cdot \hat{n}^\perp ds$. Then the harmonic vector field
\[H_{\OM} (x)=\frac{1}{2\pi}\na^\perp \ln |\Tc(x)|= \frac{1}{2\pi} D\Tc^T(x)\frac{\Tc(x)^\perp}{|\Tc(x)|^2}\]
is the unique\footnote{see e.g. \cite{ift_lop_euler}.} vector field verifying
\begin{equation*}
\diver H_{\OM} = \curl H_{\OM} =0 \text{ in } {\OM}, \quad H_{\OM}\cdot \hat{n}=0 \text{ on } \pd \Cc , \quad H_{\OM} (x)\to 0 \text{ as }|x|\to\infty, \quad \oint_{\pd \Cc} H_{\OM} \cdot \hat \t \, ds =1.
\end{equation*}
\label{1/x}Using Proposition \ref{T-inf}, we see that $H_{\OM}(x)=\mathcal{O}(1/|x|)$ at infinity. Therefore, putting together the previous properties, we obtain the existence part of the following.

\begin{proposition}\label{biot}
Let $\om \in L^{\infty}_c (\OM)$ and $\g\in \R$. If $\OM$ is an open simply connected bounded subset of $\R^2$, then there is a unique solution $u$ of
\begin{equation*}
\left\lbrace\begin{aligned}
\diver u &=0 &\text{ in } {\OM} \\
\curl u &= \om &\text{ in } {\OM} \\
u \cdot \hat{n}&=0 &\text{ on } \pd \OM \\
\end{aligned}\right.
\end{equation*}
which is given by
\begin{equation}\label{biot bd}
u(x) = K_{\OM}[\om](x).
\end{equation}

If $\Cc$ is a closed simply connected bounded subset of $\R^2$ and $\OM = \R^2\setminus \Cc$, then there is a unique solution $u$ of
\begin{equation*}
\left\lbrace\begin{aligned}
\diver u &=0 &\text{ in } {\OM} \\
\curl u &= \om &\text{ in } {\OM} \\
u \cdot \hat{n}&=0 &\text{ on } \pd \Cc \\
u(x)\to 0 &\text{ as }|x|\to\infty\\
\oint_{\pd \Cc} u \cdot \hat \t \, ds &= \g
\end{aligned}\right.
\end{equation*}
which is given by
\begin{equation}\label{biot unbd}
u(x) = K_{\OM}[\om](x) + (\g+\int \om) H_{\OM}(x).
\end{equation}
\end{proposition}
Concerning the uniqueness, we can see e.g. \cite[Lem 2.14]{kiku} (see also \cite[Prop 2.1]{ift_lop_euler}).

\medskip

We take advantage of this explicit formula to give estimates on the kernel. We introduce
\[ R[\om](x):=  \int_{\OM} \Bigl(\dfrac{(\Tc(x)-\Tc(y))^\perp}{|\Tc(x)-\Tc(y)|^2}-\dfrac{(\Tc(x)- \Tc(y)^*)^\perp}{|\Tc(x)- \Tc(y)^*|^2}\Bigl) \om(y)\, dy,\]
so that \eqref{biot unbd} reads
\[u(x) =  \frac{1}{2\pi} D\Tc^T(x) \Bigl( R[\om](x)+ (\g+\int \om)\frac{\Tc(x)^\perp}{|\Tc(x)|^2}\Bigl).\]

\begin{proposition}\label{biot est}
Let assume that $\om$ belongs  to $L^1\cap L^\infty(\OM)$. If all the angles of $\OM$ are greater than $\pi/2$, then there exist $(C,a)\in \R^+_* \times (0,1/2]$ depending only on the shape of $\OM$ such that
\[ \|R[\om] \|_{L^\infty(\OM)} \leq C (\| \om \|_{L^1}^{1/2} \| \om \|_{L^\infty}^{1/2} +\| \om \|_{L^1}^{a} \| \om \|_{L^\infty}^{1-a}+ \| \om \|_{L^1}).\]
Moreover, $R[\om]$ is continuous up to the boundary.
\end{proposition}

In the case where $\Cc$ is a Jordan arc, the uniform bound is proved in \cite[Lem 4.2]{lac_euler} and the continuity in \cite[Prop 5.7]{lac_euler}. The proof here is almost the same, except that we have to take care that $D\Tc^{-1}$ is not bounded if there is an angle less than $\pi$ (see Theorem \ref{grisvard}). For completeness, we write the details in Section \ref{sect : technical}. In this proof, we can understand why we assume that the angles are greater than $\pi/2$: we need that $\det D \Tc^{-1}$ belongs in $L^{p_0}$ for some $p_0>2$ (see Remark \ref{DT loc}).

\subsection{Existence and properties of weak solutions} \label{sect : exist}\ 

The goal of this subsection is to derive some properties about a weak solution obtained in Theorem \ref{theorem1} from  \cite{GV_lac}. We will also establish similar formulations verified by extensions on the full plane.

{\it a) Weak solution in an unbounded domain.}

We begin by the hardest case: let $\OM:=\R^2\setminus \Cc$, where $\Cc$ is a bounded, simply connected closed set, such that $\pd \Cc$ is a $C^\infty$ Jordan curve, except in a finite number of point $z_1$, $z_2$, ..., $z_n$ where $\pd \OM$ admits corner of angle $\a_i$. Then, there exists some pieces of the boundary which are smooth, implying that the capacity of $\Cc$ is greater than $0$ (see e.g. \cite[Prop 6]{GV_lac}). Therefore, Theorem \ref{theorem1} with our exterior domains is a direct consequence of \cite[Theo 2]{GV_lac}.

\medskip

We know the existence of a global weak solution. We search now some features of such a solution. Let $u_0$ satisfying \eqref{typeinitialdata} and $u$ be a global weak solution of \eqref{Euler}-\eqref{conditions} in the sense of \eqref{imperm2} and \eqref{Eulerweak} such that
$$u \in L^\infty_{\loc}(\R^+; L^2_{\loc} (\Omega)), \:  \om := \curl u \in   L^\infty(\R^+;L^1 \cap L^\infty(\Omega)).$$
As $\om_0:=\curl u_0$ is compactly supported in $\OM$ we note that we can define the initial circulation. Indeed, let $J$ be a smooth closed Jordan curve in $\OM$ such that  $\mathcal{C}$ is included in the bounded component of $\R^2 \setminus J$ and $\supp \om_0$ in the unbounded component. Therefore, we can define the real number
\[ \g_0 : = \oint_{J} u_0 \cdot \hat \t \, ds.\]
Let us remind that  $u_0$ satisfies \eqref{typeinitialdata}, so that it belongs to $W^{1,q}_{\loc}$ for all finite $q$, and so that  the integral at the r.h.s. is well-defined.  Moreover, $\g_0$ does not depend on the curve separating $\mathcal{C}$ and $\supp \om_0$ (thanks to the curl free condition near  $\mathcal{C}$). Passing to the limit, we obtain
\begin{equation}\label{g_0}
 \g_0 = \oint_{\pd \Cc} u_0 \cdot \hat \t \, ds.
\end{equation}

We have proven in the previous subsection that we can reconstruct the velocity in terms of the vorticity and the circulation:
\[u_0(x) = K_{\OM}[\om_0](x) + (\g_0+\int \om_0) H_{\OM}(x).\]

From the definition of weak solution, we know that the quantities $\| \om(t,\cdot) \|_{L^1\cap L^\infty(\OM)}$ and $\int \om(t,\cdot)$ are bounded in $\R^+$. Moreover, we infer that the circulation
\[ \g(t): = \oint_{\pd \Cc} u(t,\cdot) \cdot \hat \t \, ds\]
is bounded locally in time. To show this estimate, first, we note that the previous integral is well defined putting
\[ \g(t): = \oint_{J} u(t,\cdot) \cdot \hat \t \, ds - \int_{A} \om(t,\cdot)  \, dx, \]
with $A=\OM \cap ($bounded connected component of $\R^2 \setminus J)$. Indeed, thanks to the uniqueness part of Proposition \ref{biot} with $\oint_{J} u(t,\cdot) \cdot \hat \t \, ds = g_0$, we state that $u$ can be written as in \eqref{biot unbd}, and we deduce from Theorem \ref{grisvard} and Proposition \ref{biot est} that  $\oint_{\pd \Cc} u(t,\cdot) \cdot \hat \t \, ds$ is well defined.

Next, let $K$ be a compact subset of $\OM$. In this subset, we know by the definition of $\Tc$ and Proposition \ref{biot est} that $K_{\OM}[\om(t,\cdot)](x)$ is uniformly bounded in $\R^+\times K$. Then there exist $C_1, C_2$ such that \eqref{biot unbd} implies
\[ C_1 | \g(t)| \leq \| u(t,\cdot) \|_{L^2(K)} + C_2 + C_1 \| \om(t,\cdot) \|_{L^1(\OM)} ,\]
for any $t\in \R^+$ (we have $C_1=\| H_{\OM} \|_{L^2(K)}$). As $u$ belongs to $L^\infty_{\loc}(\R^+;L^2(K))$ (see the definition of weak solution), then we have that 
\begin{equation}\label{g bd}
\g \in L^{\infty}_{\loc}([0,\infty)).
\end{equation}

Moreover, putting together this estimate of $\g$, Remark \ref{DT loc} and Proposition \ref{biot est}, then \eqref{biot unbd} gives that
\begin{equation}\label{est u}
u \in L^{\infty}_{\loc}([0,\infty);L^p_{\loc}(\overline{\OM})), \ \forall p<4,
\end{equation}
which is an improvement compared to the definition of weak solution, because we control up to the boundary.

Let us derive a formulation verified by $\om$.

First, we note that for any test function $\f\in \Dc([0,\infty)\times \OM ; \R)$, then $\p:= \na^\perp \f$ belongs to the set of admissible test functions, and \eqref{Eulerweak} reads 
\begin{equation} \label{tourEulerweak}
\forall \, \f \in \Dc([0,\infty)\times \OM ; \R), \quad  \int_0^{\infty} \int_\Omega \left( \om \cdot \pd_t \f +   \om u \cdot \na \f \right)  = -\int_\Omega \om_0 \f(0, \cdot).
\end{equation}
Then, $(\om,u)$ verifies the transport equation
\begin{equation}\label{transport*}
 \pd_t \om + u\cdot \na \om =0
\end{equation}
in the sense of distribution \eqref{tourEulerweak} in $\OM$. We need a formulation on $\R^2$. For that, we denote by $\bar \om$ (respectively $\bar u$) the extension of $\om$ (respectively $ u$) to $\R^2$ by zero in $\OM^c$. Let us check that it verifies the transport equation for any test function $\f\in C^\infty_c(\R \times \R^2)$.

\begin{proposition}\label{tour extension}
Let $(\om,u)$ a weak solution to the Euler equations in $\OM$. Then, the pair of extension verifies in the sense of distribution
\begin{equation}
\left\lbrace \begin{aligned}
\label{tour_equa}
&\pd_t \bar\om+\bar u\cdot \na\bar\om=0, & \text{ in }\R^2\times(0,\infty) \\
& \diver\bar u=0 \text{ and }\curl\bar u=\bar\om+g_{\bar\om,\g}(s)\d_{\pd \Cc}, &\text{ in }\R^2\times[0,\infty) \\
& |\bar u|\to 0, &\text{ as }|x|\to \infty \\
& \bar \om(x,0)=\bar\om_0(x), &\text{ in }\R^2.
\end{aligned} \right .
\end{equation}
where $\d_{\pd \Cc}$ is the Dirac function along the curve and with
\begin{equation}\label{g_o_bis}
\begin{split}
g_{\bar\om,\g}(x)=& u\cdot \hat{\t}\\
=&\Bigl[ \lim_{\r\to 0^+}  K_{\OM}[\bar\om](x-\r \hat{n})+(\g+\int\bar \om)H_{\OM}(x-\r \hat{n}) \Bigl]\cdot \overrightarrow{\t}
\end{split}\end{equation}
\end{proposition}

\begin{proof} The third and fourth points are obvious. The second point is a classical computation concerning tangent vector fields: there is no additional term on the divergence, whereas it appears on the $\curl$ the jump of the tangential velocity (see e.g. the proof of Lemma 5.8 in \cite{lac_euler}).

Concerning the first point, we have to consider the case of a test function whose the support meets the boundary. Let $\f \in C^\infty_c(\R\times \R^2)$. We introduce $\F$ a non-decreasing function which is equal to 0 if $s\leq 1$ and to 1 if $s\geq 2$. Let
\[\F^{\e}(x):= \F\Bigl(\frac{|\Tc(x)|-1}{\e}\Bigl).\]
We note that
\begin{itemize}
\item it is a cutoff function of an $\e$-neighborhood of $\Cc$, because $\Tc$ is continuous up to the boundary (see Theorem \ref{grisvard}).
\item we have $\na \F^\e\cdot H_{\OM}\equiv 0$, because $H_{\OM}(x) = \frac{\na^\perp |\Tc(x)|}{|\Tc(x)|}$ (see Subsection \ref{sect : biot}).
\item the Lebesgue measure of the support of $\na \F^\e$ is $o(\sqrt{\e})$. Indeed the support of $\na \F^\e$ is contained in the subset $\{ x\in\OM_\e | 1+\e  \leq |\Tc(x)| \leq 1 + 2\e\}$. The Lebesgue measure can be estimated thanks to Remark \ref{DT loc}:
\[ \int_{ 1+\e  \leq |\Tc(x)| \leq 1 + 2\e}dx=\int_{1+\e  \leq |z| \leq 1 + 2\e} |\det(D\Tc^{-1})|(z) dz \leq \sqrt{\e} \|\det(D\Tc^{-1})\|_{L^2(B(0,1+2\e)\setminus B(0,1))},\]
where the norm in the right hand side term tends to zero as $\e\to 0$ (by the dominated convergence theorem).
\end{itemize}
Another interesting property is the fact that the velocity is tangent to the boundary whereas $\na \F^\e$ is normal. Indeed, we claim the following.
\begin{lemma}\label{ortho}
As $\om$ belongs to  $L^\infty(\R^+; L^1\cap L^\infty(\OM))$ then 
\[ u \cdot \na\F^\e \to 0 \text{ strongly in }L^1(\R^2),\]
uniformly in time, when $\e\to 0$.
\end{lemma}
\noindent
This property is not so obvious, because $|u\cdot \na \F^\e| \approx \frac{|D\Tc|^2}{\e} R[\om] \F'\Bigl(\frac{|\Tc(x)|-1}{\e}\Bigl)$ with $\| \F'\Bigl(\frac{|\Tc(x)|-1}{\e}\Bigl) \|_{L^1} = O(\e)$ (in the case where $D\Tc^{-1}$ is bounded) and $D\Tc$ blowing up.  The perpendicular argument is crucial here and we use the explicit formula \eqref{biot unbd} to show the cancellation effect. This lemma is proved in the case where $\Cc$ is a Jordan arc in \cite[Lem 4.6]{lac_euler}. For a sake of completeness, we give the general proof in Section \ref{sect : technical}.

As $\F^\e \f$ belongs to $C^{\infty}_c(\R\times \OM)$ for any $\e>0$, we can write that $(\om,u)$ is a weak solution in $\OM$:
\[ \int_0^\infty\int_{\R^2}(\F^\e\f)_t\om \, dxdt +\int_0^\infty \int_{\R^2}\na(\F^\e\f)\cdot u\om \, dxdt+\int_{\R^2}(\F^\e\f)(0,x)\om_0(x) \,dx=0.\]
As $\om\in L^\infty(L^1\cap L^\infty)$, it is obvious that the first and the third integrals converge to
\[ \int_0^\infty\int_{\R^2}\f_t \bar \om \,dxdt \text{ and } \int_{\R^2}\f(0,x)\bar \om_0(x) \, dx\]
as $\e\to 0$.
Concerning the second integral, we have
\[\int_0^\infty \int_{\R^2}\na(\F^\e\f)\cdot u\om \, dxdt = \int_0^\infty \int_{\R^2}\f (\na \F^\e \cdot u) \om \, dxdt + \int_0^\infty \int_{\R^2}\F^\e \na \f \cdot u\om \, dxdt.\]
The first right hand side term tends to zero because $\na \F^\e \cdot u \to 0 $ in $L^1(\R^2)$ and $\om\in L^\infty(\R^+\times \R^2)$. The second right hand side term converges to
\[\int_0^\infty \int_{\R^2} \na \f \cdot \bar u\bar \om \, dxdt\]
because $u$ belongs to $L^2(\supp \f \cap (\R^+_*\times \overline{\OM}))$ (see \eqref{est u}). Putting together these limits, we obtain that:
\[ \int_0^\infty\int_{\R^2}\f_t\bar\om \, dxdt +\int_0^\infty \int_{\R^2}\na\f\cdot \bar u\bar\om \, dxdt+\int_{\R^2}\f(0,x)\bar\om_0(x)\, dx=0,\]
which ends the proof.
\end{proof}

\medskip

The goal of the following is to prove that the $L^p$ norm, the total mass of the vorticity and the circulation are conserved quantities.

In a domain with smooth boundaries, the pair $(\om,u)$ is a strong solution of the transport equation, and the conservation of the previous quantities is classical. The main point here is to remark that this pair in our case is a renormalized solution in the sense of DiPerna and Lions (see \cite{dip-li}) of the transport equation. We consider equation \eqref{transport*} as a linear transport equation with given velocity field ${u}$. Our purpose here is to show that if $\bar\om$ solves this linear equation, then so does $\beta(\bar\om)$ for a suitable smooth function $\b$. This follows from the theory developed in \cite{dip-li} where they need that the velocity field belongs to $L_{\loc}^1\left(\R^+,W_{\loc}^{1,1}(\RR)\right) \cap L^1_{\loc} \left(\R^+,L^1(\R^2)+ L^\infty(\R^2)\right)$ and that $\diver u$ is bounded. Let us check that we are in this setting.

\begin{lemma}\label{W11}
Let $(\om,u)$ be a global weak solution in $\OM$. Then we have that
\[ u \in L^\infty_{\loc} \left(\R^+,W_{\loc}^{1,p}(\overline{\OM})\right)\cap L^\infty_{\loc} \left(\R^+,L^1(\overline{\OM})+ L^\infty(\overline{\OM})\right),\]
for any $p\in [1,4/3)$.
\end{lemma}

We use the explicit form of the velocity \eqref{biot unbd}: $u(x)=DT(x) f(T(x))$, where $f$ looks like the Biot-Savart operator in $\R^2$. Therefore, Lemma \ref{W11} follows from the fact that $DT$ belongs to $W^{1,p}_{\loc}(\overline{\OM})$ for any $p<4/3$ (see Theorem \ref{grisvard}), and thanks to Proposition \ref{biot est} and the Calderon-Zygmund inequality. The proof is written in \cite{lac_small} in the case where $\Cc$ is a Jordan arc. We generalize it in Section \ref{sect : technical}.

In Proposition \ref{tour extension}, we have proved that $\bar \om$ verifies the transport equation with velocity $\bar u$, but actually it verifies the transport equation with any extension $\tilde u$ of $u$ (indeed, $\bar \om\equiv 0$ outside $\OM$). Let us introduce a relevant extension of $u$ in order to apply the renormalized theory. We fix $p\in (1,4/3)$ and $R>0$ such that $\partial \OM \subset B(0,R)$. We readily check that $\widetilde\OM:= \OM\cap B(0,R+1)$ verifies the {\it Uniform Cone Condition} (see \cite[Par. 4.8]{Adams} for the precise definition). Therefore by Theorem \cite[Theo. 5.28]{Adams} for any $p\in (1,4/3)$ there exists a {\it simple} (2,p)-extension operator $E(p)$ from $W^{2,p}(\widetilde\OM)$ to $W^{2,p}(\R^2)$, namely there exists $K(p)>0$ such that for any $v \in W^{2,p}(\widetilde\OM)$ we have
\[
E(p)v=v \text{ a.e. in } \widetilde\OM, \quad \| E(p)v \|_{W^{2,p}(\R^2)} \leq K(p) \| v \|_{W^{2,p}(\widetilde\OM)}.
\]
Let us consider the stream function $\psi$ of $u$, namely the function verifying:
\[
u=\nabla^\perp \psi \text{ in } \OM, \quad \psi =0 \text{ on } \partial \Omega.
\]
By the Poincar\'e inequality, we have that
\begin{equation*}
\psi \in L^\infty_{\loc}(\R_+,W^{2,p}(\widetilde\OM)).  
\end{equation*}
Then we define $\chi$ a smooth cutoff function such that $\chi \equiv 1$ on $B(0,R)$ and $\chi \equiv 0$ on $B(0,R+1)$, and we put
\[
\tilde \psi :=  E(p)(\chi\psi),\quad \tilde u\vert_{B(0,R)} = \nabla^{\perp} \tilde \psi, \quad  \tilde u\vert_{B(0,R)^c} = u.
\]
Obviously, we have:
$$\tilde u =u \text{ a.e. in }\Omega, \quad \diver \tilde u=0 \text{ on }\R^2,$$
and 
\[\tilde u \in L^\infty_{\loc} \left(\R^+,W^{1,p}(B(0,R))\right),\]
hence
\[\tilde u \in L^\infty_{\loc} \left(\R^+,W_{\loc}^{1,1}(\RR)\right)\cap L^\infty_{\loc} \left(\R^+,L^1(\R^2)+ L^\infty(\R^2)\right).\]

Therefore, \cite{dip-li} implies that $\overline{\om}$ is a renormalized solution.

\begin{lemma}
 \label{renorm1} For $\bar u$ fixed.
Let $\overline{\om}$ be a solution of the linear equation \eqref{transport*} in $\R^2$ with velocity $\tilde u$ and initial datum $\bar\om_0$. Let $\beta:\R \rightarrow \R$ be a smooth function such that
\begin{equation*}
|\beta'(t)|\leq C(1+ |t|^p),\qquad \forall t\in \R,
\end{equation*}
for some $p\geq 0$. Then $\beta(\bar\om)$ is a solution of \eqref{transport*} in $\R^2$ (in the sense of distribution) with velocity $\tilde u$ and initial datum $\beta(\bar\om_0)$.
\end{lemma}
We recall that  $\bar \om$ denotes the extension of $\om$ by zero in $\Cc$, and the previous lemma means that, for any $\F\in C^\infty_c([0,\infty)\times \R^2)$ and $\beta$ such that $\beta(0)=0$, we have 
\begin{equation}\label{renorm}
\frac{d}{dt} \int_{\R^2} \beta(\om)\F(t,x)\, dx=\int_{\R^2} \beta(\om) (\pd_t \F +u\cdot \na \F)\,dx
\end{equation}
in the sense of distributions on $\R^+$. Now, we write a remark from \cite{lac_miot} in order to establish some desired properties for $\om$. 

\begin{remark}
\label{remark : conserv} (1) Since the right-hand side in \eqref{renorm} belongs to $L_{\loc}^1(\R^+)$, the equality holds in $L_{\loc}^1(\R^+)$. With this sense, \eqref{renorm} actually still holds when $\F$ is smooth, bounded and has bounded first derivatives in time and space. In this case, we have to consider smooth functions $\beta$ which in addition satisfy $\beta(0)=0$, so that $\beta(\om)$ is integrable. This may be proved by approximating $\F$ by smooth and compactly supported functions $\F_n$ for which \eqref{renorm} applies, and by letting then $n$ go to $+\infty$.\\
(2) We apply the point (1) for $\beta(t)=t$ and $\F\equiv 1$, which gives
\begin{equation}\label{om-est-1}
 \int_{\OM} \om(t,x)\, dx =   \int_{\OM} \om_0(x)\, dx \text{ for all }t>0.
 \end{equation} 
(3) We let $1\leq p<+\infty$. Approximating $\beta(t)=|t|^p$ by smooth functions and choosing $\F\equiv 1$ in \eqref{renorm}, we deduce that for a solution $\om$ to \eqref{transport*}, the maps $t\mapsto \|\om(t)\|_{L^p(\OM)}$ are continuous and constant. In particular, we have
\begin{equation}\label{om-est-2}
 \|\om(t)\|_{L^1(\OM)}+\|\om(t)\|_{L^\infty(\OM)}\equiv \|\om_0\|_{L^1(\OM)}+\|\om_0\|_{L^\infty(\OM)}.
\end{equation}
\end{remark}

In the case of unbounded domain, we will require that $\om$ stays compactly supported. Specifying our choice for $\F$ in \eqref{renorm}, we are led to the following.

\begin{proposition}\label{compact_vorticity} 
Let $\om$ be a weak solution of \eqref{transport*} such that
\begin{equation*}
\om_0  \text{ is compactly supported in } B(0,R_0)
\end{equation*}
for some positive $R_0$. For any $\T^*$ fixed, then there exists $C>0$ such that
\begin{equation*}
\om(t,\cdot)  \text{ is compactly supported in } B(0,R_0+Ct),
\end{equation*}
for any $t\in [0,\T^*]$.
\end{proposition}

The main computation of this proof can be found in \cite{lac_miot} or in \cite{lac_small}. For a sake of self-containedness we write the details in Section \ref{sect : technical}.

\medskip

Therefore, for $\T^*$ fixed, there exists $R_1$ such that the support of the vorticity is included in $B(0,R_1)$ for all $t\in [0,\T^*]$. It implies that $u$ is harmonic in $B(0,R_1)^c$ ($\diver u = \curl u =0$), and \eqref{Euler} is verified in the strong way on this set. With strong solution, the Kelvin's circulation theorem can be used, which states that the circulation at infinity is conserved:
\[ \g(t)+\int \om(t,\cdot) = \g^\infty(t) \equiv \g^\infty_0 = \g_0 + \int \om_0.\]
Using the conservation of the total mass \eqref{om-est-1}, we obtain that the circulation of the velocity around the obstacle is conserved:
\begin{equation}\label{gamma cons}
 \g(t) \equiv \g_0,\ \forall t \in [0,\T^*].
\end{equation}

{\it b) Weak solution in a bounded domain.}

The previous part can be adapted easily to the bounded case. In simply connected domain, we do not consider the circulation:
\[ u_0(x) = K_{\OM}[\om_0].\]
As Proposition \ref{tour extension} is about the behavior near the boundary, we can check that we obtain exactly the same.

\begin{proposition}\label{tour extension bd}
Let $(\om,u)$ a weak solution to the Euler equations in $\OM$ bounded. Then, the pair of extension verifies in the sense of distribution
\begin{equation}
\left\lbrace \begin{aligned}
\label{tour_equa_bis bd}
&\pd_t \bar\om+\bar u\cdot \na\bar\om=0, & \text{ in }\R^2\times(0,\infty) \\
& \diver\bar u=0 \text{ and }\curl\bar u=\bar\om+g_{\bar\om}(s)\d_{\pd \OM}, &\text{ in }\R^2\times[0,\infty) \\
& \bar \om(x,0)=\bar\om_0(x), &\text{ in }\R^2.
\end{aligned} \right .
\end{equation}
where $\d_{\pd \OM}$ is the Dirac function along the curve and $g_{\bar\om}$ is :
\begin{equation}\label{g_o_bis bd}
\begin{split}
g_{\bar\om}(x)=& -u\cdot \hat{\t}\\
=&-\Bigl[ \lim_{\r\to 0^+}  K_{\OM}[\bar\om](x - \r \hat{n}) \Bigl]\cdot \hat{\t}
\end{split}\end{equation}
\end{proposition}

Moreover, we have a term less compared of the unbounded case, then we can also check that 
\[ u \in L^\infty_{\loc} \left(\R^+,W^{1,p}({\OM})\right)\cap L^\infty_{\loc} \left(\R^+,L^1({\OM})\right),\]
for any $p\in [1,4/3)$. As in the unbounded case, we fix $p\in (1,4/3)$, $R>0$ such that $\partial \OM \subset B(0,R)$, and $\chi$ a smooth cutoff function such that $\chi \equiv 1$ on $B(0,R)$ and $\chi \equiv 0$ on $B(0,R+1)$. We also consider the stream function $\psi$ of $u$:
\[
u=\nabla^\perp \psi \text{ in } \OM, \quad \psi =0 \text{ on } \partial \Omega.
\]
which verifies by the Poincar\'e inequality
\begin{equation*}
\psi \in L^\infty_{\loc} (\R_+,W^{2,p}(\OM)).  
\end{equation*}
We put
\[
\tilde \psi := \chi E(p)\psi,\quad \tilde u = \nabla^{\perp} \tilde \psi.
\]
Obviously, we have:
$$\tilde u =u \text{ a.e. in }\Omega, \quad \diver \tilde u=0 \text{ on }\R^2.$$
As $\tilde u$ is compactly supported in $B(0,R+1)$ we infer that
\begin{equation*}
\tilde u \in L^\infty_{\loc} (\R_+;W^{1,p}(\R^2)),
\end{equation*}
hence
\[\tilde u \in L^\infty_{\loc} \left(\R^+,W_{\loc}^{1,1}(\RR)\right)\cap L^\infty_{\loc} \left(\R^+,L^1(\R^2)\right).\]

Therefore, $\overline{\om}$ is a renormalized solution and that 
\begin{equation}\label{om-est-1 bd}
 \int_{\OM} \om(t,x)\, dx =   \int_{\OM} \om_0(x)\, dx \text{ for all }t>0
 \end{equation} 
and
\begin{equation}\label{om-est-2 bd}
 \|\om(t)\|_{L^1(\OM)}+\|\om(t)\|_{L^\infty(\OM)}\equiv \|\om_0\|_{L^1(\OM)}+\|\om_0\|_{L^\infty(\OM)}.
\end{equation}

\section{Liapounov method}\label{sect : 3}

In this section, we present the proof for a Lagrangian solution. When the velocity $u$ is smooth, it
gives rise to a flow $\phi_x(t)$ defined by
\begin{equation}
\label{i2}
\begin{cases}
\frac{d}{dt} \phi_x(t)=u\big(t,\phi_x(t)\big) \\
\phi_x(0)=x \in \RR.
\end{cases}
\end{equation}
In view of \eqref{transport}, we then have
\begin{equation}
\label{i3} \frac{d}{dt} \om\big(t,\phi_x(t)\big)\equiv 0,
\end{equation}
which gives that $\om$ is constant along the characteristics. We assume here that these trajectories exist and are differentiable in our case, and we prove by the Liapounov method that the support of the vorticity never meets the boundary $\pd\OM$. Although we do not know that the flow is smooth, the following computation is the main idea of this article, and it will be rigourously applied in Section \ref{sect : 4}.

The Liapounov method to prove this kind of result is used by Marchioro and Pulvirenti in \cite{mar_pul} in the case of a point vortex which moves under the influence of the regular part of the vorticity, and by Marchioro in \cite{mar} when the dirac is fixed. In both articles, the authors use the explicit formula of the velocity associated to the dirac centered at $z(t)$: $H(x)=(x-z)^\perp/(2\pi |x-z|^2)$. The geometrical structure is the key of their analysis. Indeed, choosing $L(t) = - \ln |\phi_x(t)-z(t) |$ they have that
\begin{enumerate}
\item[a)] $L(t) \to \infty$ if and only if the trajectory meets the dirac. Then, it is sufficient to prove that $L'(t)$ stays bounded in order to prove the result.
\item[b)] $H( \phi_x(t)) \cdot( \phi_x(t)-z(t)) \equiv 0$, which implies that the singular term in the velocity does not appear.
\end{enumerate}

Therefore, the explicit blow up in the case of the dirac point is crucial in two points of view: for the symmetry cancelation (point b) and for the fact that the primitive of $1/x$ is $\ln x$ which blows up near the origin (point a). In our case, we do not have such an explicit form of the blow up near the corners and the primitive of $1/\sqrt{x}$ is $\sqrt{x}$ which is bounded near $0$. The idea is to add a logarithm. When $\Cc$ is a Jordan arc, $|\Tc| \approx 1+\sqrt{z^2-1}$ and we note that $\ln \ln (1+ \sqrt{z^2 -1})$ blows up near the end-points $\pm 1$.

However, the problem with Liapounov function is that it is very specific on the case studied. For example, this function is different if the dirac point is fixed or if it moves with the fluid (for more details and explanations, see the discussion on Liapounov functions in Section \ref{sect : 6}).

\bigskip

We fix $x_0\in \OM$ and we consider $\phi=\phi_{x_0}(t)$ the trajectory which comes from $x_0$ (see \eqref{i2}). We denote
\[ L(t):= -\ln |L_1(t, \phi(t)) |\]
with $L_1$ depending on the geometric property of $\OM$:\label{assump}
\begin{enumerate}
\item if $\Omega$ is a bounded, simply connected open set, such that $\pd \OM$ has a finite number of corner with angles greater than $\pi/2$ (as in Theorem \ref{main 1}), then we choose
\begin{equation}\label{L1 1}
 L_1(t,x):= \frac1{2\pi} \int_{\OM} \ln\Bigl( \frac{|\Tc(x)-\Tc(y)|}{|\Tc(x)-\Tc(y)^*||\Tc(y)|} \Bigl)\om(t,y) \, dy;
\end{equation} 
\item if $\OM:=\R^2\setminus \Cc$, where $\Cc$ is a compact, simply connected set, such that $\pd \OM$ has a finite number of corner with angles greater than $\pi/2$ (as in Theorem \ref{main 2}), then we choose
\begin{equation}\label{L1}
 L_1(t,x):= \frac1{2\pi} \int_{\OM} \ln\Bigl( \frac{|\Tc(x)-\Tc(y)|}{|\Tc(x)-\Tc(y)^*||\Tc(y)|} \Bigl)\om(t,y) \, dy+ \frac{\a}{2\pi}\ln |\Tc(x)|,
\end{equation} 
where $\a:= \g_0+\int \om_0$.
\end{enumerate}

When trajectories exist, it is obvious (without renormalization) that \eqref{i2} and \eqref{i3} imply that
\begin{equation}\label{norm om}
\| \om(t,\cdot)\|_{L^p} =\| \om_0\|_{L^p} \text{ and } \int_{\OM} \om(t,\cdot) = \int_{\OM} \om_0,\ \forall t>0, \forall p\in[1,\infty].
\end{equation}
We assume that $\om_0$ is compactly supported, then included in $B(0,R_0)$ for some $R_0>0$. Thanks to Propositions \ref{T-inf} and \ref{biot est}, we see that the velocity $u$ is bounded outside this ball by a constant $C_0$, and  \eqref{i2} and \eqref{i3} give
\begin{equation}\label{support}
\supp \om(t,\cdot) \subset B(0,R_0+C_0 t),\ \forall t\geq 0.
\end{equation}
We also have that the circulation is conserved.

If we assume that $\om_0$ is non positive, then it follows from \eqref{i3} that
\begin{equation}\label{signe}
 \om(t,x) \leq 0,\ \forall t\geq 0,\ \forall x \in \OM.
\end{equation}

\subsection{Blow up of the Liapounov function near the boundary.}\ 

The first required property is that $L$ goes to infinity iff the trajectory meets the boundary. Next, if we prove that $L$ is bounded, then it will follow that the trajectory stays far away the boundary. 
We fix $\mathbf{T}^*>0$, using \eqref{support} we denote by $R_{\mathbf{T}^*}:= R_0+C_0 {\mathbf{T}^*}$, such that $\supp \om(t,\cdot)\subset B(0,R_{\mathbf{T}^*})$ for all $t\in [0,{\mathbf{T}^*}]$. 

\begin{lemma}\label{L1 maj} 
For any case (1)-(2), there exists $C_1=C_1({\mathbf{T}^*},\om_0,\g_0)$ such that
\[ |L_1(t,x)| \leq C_1||\Tc(x)|-1|^{1/2}, \ \forall x\in B(0,R_{\mathbf{T}^*}), \ \forall t\in [0,{\mathbf{T}^*}].\]
\end{lemma}
\begin{proof} 
For a sake of shortness, we write the proof in the hardest case: case (2). The other case follows easily. Recalling the notation $z^*=z/|z|^2$, we can compute
\begin{equation*}\begin{split}
\frac{|\Tc(x)-\Tc(y)|^2}{|\Tc(x)-\Tc(y)^*|^2|\Tc(y)|^2} &= 1-\frac{ |\Tc(x)-\Tc(y)^*|^2|\Tc(y)|^2-|\Tc(x)-\Tc(y)|^2}{ |\Tc(x)-\Tc(y)^*|^2|\Tc(y)|^2}\\
= 1-&\frac{( |\Tc(x)|^2 |\Tc(y)|^2 -2 \Tc(x)\cdot \Tc(y) +1)- (|\Tc(x)|^2 -2\Tc(x)\cdot \Tc(y) + |\Tc(y)|^2)}{ |\Tc(x)-\Tc(y)^*|^2|\Tc(y)|^2}\\
&= 1- \frac{ (|\Tc(x)|^2-1)(|\Tc(y)|^2-1)}{|\Tc(x)-\Tc(y)^*|^2|\Tc(y)|^2}.
\end{split}\end{equation*}
Therefore, we have
\begin{eqnarray*}
\ln\Bigl( \frac{|\Tc(x)-\Tc(y)|}{|\Tc(x)-\Tc(y)^*||\Tc(y)|}\Bigl) &=& \frac12 \ln\Bigl( \frac{|\Tc(x)-\Tc(y)|^2}{|\Tc(x)-\Tc(y)^*|^2|\Tc(y)|^2}\Bigl) \\
&=& \frac12 \ln\Bigl(1- \frac{ (|\Tc(x)|^2-1)(|\Tc(y)|^2-1)}{|\Tc(x)-\Tc(y)^*|^2|\Tc(y)|^2}\Bigl),
\end{eqnarray*}
and we need an estimate of $\ln(1-r)$ when $r\in (0,1)$, because we recall that $|\Tc(z)| > 1$ for any $z\in \OM$. It is easy to see (studying the difference of the functions) that
\[ |\ln(1-r)| = - \ln(1-r)\leq  \Bigl(\frac{r}{1-r}\Bigl)^{1/2},\ \forall r\in [0,1).\]
Applying this inequality, we have for any $y\neq x$ that 
\begin{eqnarray*}
\Bigl| \ln\Bigl( \frac{|\Tc(x)-\Tc(y)|}{|\Tc(x)-\Tc(y)^*||\Tc(y)|}\Bigl)\Bigl| &\leq & \frac{1}2  \Bigl(\frac{\frac{ (|\Tc(x)|^2-1)(|\Tc(y)|^2-1)}{|\Tc(x)-\Tc(y)^*|^2|\Tc(y)|^2}}{\frac{|\Tc(x)-\Tc(y)|^2}{|\Tc(x)-\Tc(y)^*|^2|\Tc(y)|^2} }\Bigl)^{1/2}\\
&\leq& \frac{1}2\frac{ \sqrt{(|\Tc(x)|^2-1)(|\Tc(y)|^2-1)}}{|\Tc(x)-\Tc(y)|}.
\end{eqnarray*}
By continuity of $\Tc$, we denote by $C_{\mathbf{T}^*}$ a constant such that $\Tc(B(0,R_{\mathbf{T}^*}))\subset B(0,C_{\mathbf{T}^*})$. Finally, we apply the previous inequality to $L_1$ and we find for all $x\in B(0,R_{\mathbf{T}^*})$ and $t\in [0,{\mathbf{T}^*}]$:
\begin{eqnarray*}
 |L_1(t,x)| & \leq & \frac{C_{\mathbf{T}^*}(C_{\mathbf{T}^*}+1)^{1/2}}{4\pi}(|\Tc(x)|-1)^{1/2} \int_{\OM} \frac{|\om(y)|}{|\Tc(x)-\Tc(y)|} \, dy + \frac{|\a|}{2\pi}\ln |\Tc(x)|\\
 &\leq & \frac{\sqrt{2}C_{\mathbf{T}^*}^{3/2}}{4\pi}(|\Tc(x)|-1)^{1/2} C (\| \om \|_{L^1}^{1/2} \| \om \|_{L^\infty}^{1/2} + \| \om \|_{L^1}^{a} \| \om \|_{L^\infty}^{1-a}) + \frac{|\a|}{2\pi} ( |\Tc(x)| -1).
\end{eqnarray*}
For the last inequality, we used a part of Proposition \ref{biot est}. As $( |\Tc(x)| -1)\leq C_{\mathbf{T}^*}^{1/2} ( |\Tc(x)| -1)^{1/2}$, the conclusion follows from \eqref{norm om}.
\end{proof}

Concerning the lower bound for the case (1)-(2), we need some conditions on the sign.

\begin{lemma} \label{L1 est}
If $\om_0$ is non-positive and $\g_0\geq - \int \om_0 $ (only in case (2)), then  there exists $C_2=C_2({\mathbf{T}^*},\om_0)$ such that
\[ L_1(t,x) \geq C_2||\Tc(x)|-1|, \ \forall x\in B(0,R_{\mathbf{T}^*}), \ \forall t\in [0,{\mathbf{T}^*}].\]
\end{lemma}
\begin{proof}
Again, we write the details in the case (2). We denote by $r_\infty:= \|\om_0 \|_{L^\infty}$ and $r_1:= \|\om_0 \|_{L^1}$. For $\rho >0$, we denote by 
\[V_1:= (\Cc + B(0,\rho))\cap \OM =\{ x \in \OM; \mathrm{dist}(x,\Cc) < \rho\}, \ V_2:= \OM \setminus V_1.\]
We fix $\rho$ such that the lebesgue measure of $V_1$ is equal to $r_1/(2r_\infty)$.

We deduce from \eqref{norm om} that
\[ r_1 = \| \om(t,\cdot)\|_{L^1(V_1)} + \| \om(t,\cdot)\|_{L^1(V_2)}\]
with   $\| \om(t,\cdot)\|_{L^1(V_1)} \leq r_\infty r_1/(2r_\infty)=r_1/2$ which implies that $\| \om(t,\cdot)\|_{L^1(V_2)} \geq r_1/2$.

As the logarithm of the fraction is negative (see the proof of Lemma \ref{L1 maj}), we have, with the sign condition, that:
\[ L_1(t,x) \geq \frac1{2\pi} \int_{V_2} \ln\Bigl( \frac{|\Tc(x)-\Tc(y)|}{|\Tc(x)-\Tc(y)^*||\Tc(y)|} \Bigl)\om(y) \, dy.\]
Moreover, thanks to the computation made in the proof of Lemma \ref{L1 maj}, we have
\begin{eqnarray*}
\ln\Bigl( \frac{|\Tc(x)-\Tc(y)|}{|\Tc(x)-\Tc(y)^*||\Tc(y)|}\Bigl) &=& \frac12 \ln\Bigl( \frac{|\Tc(x)-\Tc(y)|^2}{|\Tc(x)-\Tc(y)^*|^2|\Tc(y)|^2}\Bigl) \\
&=& \frac12 \ln\Bigl(1- \frac{ (|\Tc(x)|^2-1)(|\Tc(y)|^2-1)}{|\Tc(x)-\Tc(y)^*|^2|\Tc(y)|^2}\Bigl)\\
&\leq & - \frac12 \frac{ (|\Tc(x)|^2-1)(|\Tc(y)|^2-1)}{|\Tc(x)-\Tc(y)^*|^2|\Tc(y)|^2}
\end{eqnarray*}
because $\ln(1+x)\leq x$ for any $x>-1$.

As $\rho >0$ and $\Tc$ is continuous, there exists $C_{\rho}>0 $ such that $|\Tc(y)| \geq 1+ C_{\rho}$, for all $y\in V_2$. Moreover, there exists also $\tilde R_{\mathbf{T}^*}>1$ such that $\Tc(B(0,R_{\mathbf{T}^*}))\subset B(0, \tilde R_{\mathbf{T}^*}).$  Adding the fact that $\om$ is non positive, we have for all $y\in V_2\cap \supp \om$ and $x\in B(0,R_{\mathbf{T}^*})$
\begin{eqnarray*}
\ln\Bigl( \frac{|\Tc(x)-\Tc(y)|}{|\Tc(x)-\Tc(y)^*||\Tc(y)|}\Bigl) \om(y) &\geq &  \frac12 \frac{ (|\Tc(x)|^2-1)(|\Tc(y)|^2-1)}{|\Tc(x)-\Tc(y)^*|^2|\Tc(y)|^2} |\om (y)|\\
&\geq & \frac12 \frac{ (|\Tc(x)|-1)(|\Tc(x)|+1) (|\Tc(y)|-1)(|\Tc(y)|+1)}{(|\Tc(x)|+ 1)^2 |\Tc(y)|^2} |\om (y)|\\
&\geq & \frac12 \frac{ (|\Tc(x)|-1) C_{\rho}}{( \tilde R_{\mathbf{T}^*} + 1)  \tilde R_{\mathbf{T}^*}} |\om (y)|.
\end{eqnarray*}
Integrating this last inequality over $V_2$, we obtain that
\begin{eqnarray*}
 L_1(t,x) &\geq& \frac1{2\pi} \int_{V_2} \ln\Bigl( \frac{|\Tc(x)-\Tc(y)|}{|\Tc(x)-\Tc(y)^*||\Tc(y)|} \Bigl)\om(y) \, dy \geq   \frac{ (|\Tc(x)|-1) C_{\rho}}{ 4\pi( \tilde R_{\mathbf{T}^*} + 1)  \tilde R_{\mathbf{T}^*}} \|\om\|_{L^1(V_2)} \\
 &\geq&\frac{  C_{\rho}}{ 8\pi( \tilde R_{\mathbf{T}^*} + 1)  \tilde R_{\mathbf{T}^*}} r_1 (|\Tc(x)|-1),
\end{eqnarray*}
which ends the proof.
\end{proof}

Multiplying by $-1$ the expression of $L_1$, we can establish the same result with opposite signe condition:
\begin{remark}\label{rem : L1 est}
If $\om_0$ is non-negative and $\g_0\leq -\int \om_0 $, , then  there exists $C_2$ such that
\[ -L_1(t,x) \geq C_2 ||\Tc(x)|-1|, \ \forall x\in B(0,R_{\mathbf{T}^*}), \ \forall t\in [0,{\mathbf{T}^*}].\]
\end{remark}

From these two lemmas, it follows obviously the following.

\begin{corollary}\label{lem : sign L1} 
If $\om_0$ is non-positive and $\g_0\geq -\int \om_0$ (only for (2)), then we have that 
\begin{itemize}
\item $L_1(x) > 0$ for all $x\in \OM$;
\item $L_1(x) \to 0$ if and only if $x\to \pd \OM$.
\end{itemize}
If $\om_0$ is non-negative and $\g_0\leq - \int \om_0$ (only for (2)), then we have that 
\begin{itemize}
\item $L_1(x) < 0$ for all $x\in \OM$;
\item $L_1(x) \to 0$ if and only if $x\to \pd \OM$.
\end{itemize}
\end{corollary}

Indeed, $|\Tc(x)| \to 1$ iff $x\to \pd \OM$.

\subsection{Estimates of the Liapounov}\

The issue of this part is to prove that the trajectory never meets the obstacle in finite time. In other word, let $x_0 \in \supp \om_0$ (then $L_1(0,x_0) \neq 0$) and ${\mathbf{T}^*}>0$, we will prove that $L(t)$ stays bounded in $[0,{\mathbf{T}^*}]$.  Then, we differentiate $L$:
\[ L'(t)= -\Bigl( \pd_t L_1(t,\phi (t)) + \phi'(t)\cdot \na L_1(t,\phi (t)) \Bigl) / L_1(t,\phi (t))\]
and we want to estimate the right hand side term.

As usual, we write the details for the case (2).

On one hand, we note that
\begin{eqnarray*}
u(t,x) \cdot \na L_1(t,x) &= &u(t,x) \cdot \Bigl[ \frac1{2\pi} \int_{\OM}\Bigl( \frac{\Tc(x)-\Tc(y)}{|\Tc(x)-\Tc(y)|^2}- \frac{\Tc(x)-\Tc(y)^*}{|\Tc(x)-\Tc(y)^*|^2} \Bigl)\om(y) \, dy D\Tc(x) \\
&&+ \frac{\a}{2\pi} \frac{\Tc(x)}{|\Tc(x)|^2} D\Tc(x)\Bigl] \\
&\equiv &0 
\end{eqnarray*}
thanks to the explicit formula of $u$ (see \eqref{biot unbd}).

On the other hand, we use the equation\footnote{to justify that it works even for a weak solution, the reader can see the first lines of the proof of Proposition \ref{compact_vorticity}.} verified by $\om$ to have
\begin{eqnarray*}
\pd_t L_1(t,x) &=& \frac1{2\pi} \int_{\OM} \ln\Bigl( \frac{|\Tc(x)-\Tc(y)|}{|\Tc(x)-\Tc(y)^*||\Tc(y)|} \Bigl)\pd_t \om(y) \, dy \\
&=&-\frac1{2\pi} \int_{\OM} \ln\Bigl( \frac{|\Tc(x)-\Tc(y)|}{|\Tc(x)-\Tc(y)^*||\Tc(y)|} \Bigl) \diver(u(y)\om(y)) \, dy \\
&=& \frac1{2\pi} \int_{\OM} \na_y \Bigl[\ln\Bigl( \frac{|\Tc(x)-\Tc(y)|}{|\Tc(x)-\Tc(y)^*||\Tc(y)|} \Bigl)\Bigl]\cdot u(y)\om(y) \, dy .
\end{eqnarray*}
Now, we use the symmetry of the Green kernel (see Subsection \ref{sect : biot})
\[\na_y \ln\Bigl( \frac{|\Tc(x)-\Tc(y)|}{|\Tc(x)-\Tc(y)^*||\Tc(y)|} \Bigl) = \na_y \ln\Bigl( \frac{|\Tc(y)-\Tc(x)|}{|\Tc(y)-\Tc(x)^*||\Tc(x)|} \Bigl)\]
 and the  explicit formula of $u(y)$ to write
\begin{equation*}\begin{split}
\pd_t L_1(t,x) = \frac1{2\pi} \int_{\OM}  \Bigl[ \Bigl( \frac{\Tc(y)-\Tc(x)}{|\Tc(y)-\Tc(x)|^2} &- \frac {\Tc(y)-\Tc(x)^*}{|\Tc(y)-\Tc(x)^*|^2} \Bigl) \\
&D\Tc(y)\frac{1}{2\pi} D\Tc^T(y) \Bigl( R[\om](y)+ \a \frac{\Tc(y)^\perp}{|\Tc(y)|^2}\Bigl) \Bigl] \om(y)\, dy.
\end{split}\end{equation*}
As $\Tc$ is  holomorphic, $D\Tc$ is of the form 
$\begin{pmatrix}
a & b \\
-b & a
\end{pmatrix}$
 and we can check that $D\Tc(y)D\Tc^T(y)=(a^2+b^2)Id=|\det(D\Tc)(y)|Id$, so
\begin{equation}\label{dtL1}\begin{split}
\pd_t L_1(t,x) = \frac1{(2\pi)^2} \int_{\OM}  \Bigl[ \Bigl( \frac{\Tc(y)-\Tc(x)}{|\Tc(y)-\Tc(x)|^2} &- \frac {\Tc(y)-\Tc(x)^*}{|\Tc(y)-\Tc(x)^*|^2} \Bigl) \\
& \Bigl( R[\om](y)+ \a \frac{\Tc(y)^\perp}{|\Tc(y)|^2}\Bigl) \Bigl] |\det(D\Tc)(y)| \om(y)\, dy.
\end{split}\end{equation}

The goal is to estimate $\pd_t L_1/L_1$. However, Corollary \ref{lem : sign L1} states that $L_1$ goes to zero if and only if $x\to \pd \OM$. Then it is important to show that $\pd_t L_1$  tends to zero as $x\to \pd \OM$, and to prove that it goes to zero faster than $L_1$.

We will need the following general lemma.

\begin{lemma}\label{technic}
Let $h$ be a bounded function, compactly supported in $B(0,R_h)$ for some $R_h>1$. Then, there exists $C_h=C(\| h \|_{L^\infty},R_h)$ such that
\[ \int_{D^c}\frac{|h(y)|}{|y-x||y-x^*|}\, dy \leq C_h \Bigl(|\ln (|x|-1)| +|x| \Bigl), \ \forall x \in D^c \]
with the notation $x^*=x/|x|^2$ and $D=B(0,1)$.
\end{lemma}
\begin{proof}
We fix $x\in D^c$ and we denote
\[ \rho = |x|-1 \text{ and } \rho^*=1-|x^*|=1-\frac{1}{1+\r}=\frac{\r}{1+\r} .\]

We compute
\[ \int_{D^c}\frac{|h(y)|}{|y-x||y-x^*|}\, dy = \int_{D^c\cap B(x,4\r)}\frac{|h(y)|}{|y-x||y-x^*|}\, dy +\int_{D^c\cap B(x,4\r)^c}\frac{|h(y)|}{|y-x||y-x^*|}\, dy =: I_1+I_2.\]

For $I_1$, we know that $|y-x^*| \geq |y|-|x^*| \geq \r^*$, hence
\begin{eqnarray*}
I_1 &\leq & \frac1{\r^*} \int_{D^c\cap B(x,4\r)}\frac{|h(y)|}{|y-x|}\, dy \leq \frac{\|h\|_{L^\infty}}{\r^*} \int_{B(x,4\r)}\frac{1}{|y-x|}\, dy\\
&\leq & \frac{(1+\r)\|h\|_{L^\infty}}{\r}2\pi4\r
\end{eqnarray*}
which gives that $I_1\leq C_1 |x|$.

Concerning $I_2$, we note that
\[ |x-x^*| = \r + \r^* = \r + \frac{\r}{1+\r}\leq 2 \r \leq \frac12 |y-x|\]
for any $y\in B(x,4\r)^c$. Hence,
\[ |y-x^*| \geq |y-x| - |x-x^*| \geq \frac12 |y-x|,\]
and we have
\begin{eqnarray*}
I_2 &\leq&  \int_{D^c\cap B(x,4\r)^c}\frac{2|h(y)|}{|y-x|^2}\, dy\leq 2\|h\|_{L^\infty} \int_0^{2\pi} \int_{4\r}^{|x|+R_h} \frac1r\, drd\theta\\
&\leq&  4\pi \|h\|_{L^\infty} \ln\frac{|x|+R_h}{4\r}
\end{eqnarray*}
which implies that $I_2 \leq C_2 \Bigl( |\ln(|x|-1) | + \ln \frac{|x|+R_h}{4}\Bigl)$.

We conclude because there exists $C_3 = C_3(R_h)$ such that $\ln \frac{|x|+R_h}{4} \leq C_3 |x|$ for any $x\in D^c$.
\end{proof}

We recall that we have fixed ${\mathbf{T}^*}>0$ and $x_0\in \supp \om_0$. Using \eqref{support}, we denote by $R_{\mathbf{T}^*}:= R_0+C_0 {\mathbf{T}^*}$, such that $\supp \om(t,\cdot)\subset B(0,R_{\mathbf{T}^*})$ for all $t\in [0,{\mathbf{T}^*}]$. Finally, we estimate $\pd_t L_1$ without sign conditions.

\begin{lemma} \label{dtL1 est}
There exists $C_3=C_3({\mathbf{T}^*})$ such that
\[ |\pd_t L_1(t,x) | \leq C_3||\Tc(x)|-1| \Bigl(1+ \Bigl|\ln||\Tc(x)|-1|\Bigl|\Bigl), \ \forall x\in B(0,R_{\mathbf{T}^*}), \ \forall t\in [0,{\mathbf{T}^*}].\]
\end{lemma}

\begin{proof}
Using \eqref{frac} we know that
\[\Bigl| \frac{\Tc(y)-\Tc(x)}{|\Tc(y)-\Tc(x)|^2} - \frac {\Tc(y)-\Tc(x)^*}{|\Tc(y)-\Tc(x)^*|^2} \Bigl| = \frac {|\Tc(x)-\Tc(x)^*|}{|\Tc(y)-\Tc(x)| |\Tc(y)-\Tc(x)^*|}. \]
Then, Proposition \ref{biot est} and \eqref{norm om} allow us to estimate \eqref{dtL1}
\begin{equation*}
|\pd_t L_1(t,x) | \leq C|\Tc(x)-\Tc(x)^*| \int_{\OM}  \frac { |\om(y)|}{|\Tc(y)-\Tc(x)| |\Tc(y)-\Tc(x)^*|}   |\det(D\Tc)(y)|\, dy.
\end{equation*}
On one hand, we have for all $x\in B(0,R_{\mathbf{T}^*})$
\begin{eqnarray*}
|\Tc(x)-\Tc(x)^*| &=& \frac{\Bigl|\Tc(x)|\Tc(x)|^2 - \Tc(x)\Bigl|}{|\Tc(x)|^2}= \frac{|\Tc(x)|^2 - 1}{|\Tc(x)|}\\
&=&\frac{(|\Tc(x)| - 1)(|\Tc(x)| + 1)}{|\Tc(x)|}\leq 2 (|\Tc(x)| - 1).
\end{eqnarray*}

On the other hand, we change variables $\y=\Tc(y)$ and we compute
\[\int_{\OM}  \frac { |\om(y)|}{|\Tc(y)-\Tc(x)| |\Tc(y)-\Tc(x)^*|}   |\det(D\Tc)(y)|\, dy = \int_{D^c}  \frac { |\om(\Tc^{-1}(\y))|}{|\y-\Tc(x)| |\y-\Tc(x)^*|}   \, d\y.\]
As $\| \om \circ \Tc^{-1} \|_{L^\infty} = \|\om_0 \|_{L^\infty}$ and as
\[\supp \om \circ \Tc^{-1} = \Tc(\supp \om) \subset \Tc(B(0,R_{\mathbf{T}^*}))\subset B(0,\tilde R_{\mathbf{T}^*}),\]
we apply Lemma \ref{technic} to establish that
\[\int_{\OM}  \frac { |\om(y)|}{|\Tc(y)-\Tc(x)| |\Tc(y)-\Tc(x)^*|}   |\det(D\Tc)(y)|\, dy \leq C \Bigl(|\ln (|\Tc(x)|-1)| + \tilde R_{\mathbf{T}^*} \Bigl), \ \forall x \in B(0,R_{\mathbf{T}^*}).\]
This finishes the proof.
\end{proof}

\begin{remark}
In the bounded case, there is a tricky difference in the previous proof. We note that
\[ |\Tc(x)-\Tc(x)^*| =\frac{(1- |\Tc(x)|)(|\Tc(x)| + 1)}{|\Tc(x)|}\leq 2 \frac{(1-|\Tc(x)| )}{|\Tc(x)|}\]
with $|\Tc(x)|$ which can go to zero. To fix this problem, we can prove a similar result to  Lemma \ref{technic}: there exists $C_h=C(\| h \|_{L^\infty})$ such that
\[ \frac1{|x|} \int_{D}\frac{|h(y)|}{|y-x||y-x^*|}\, dy \leq C_h \Bigl(|\ln (1-|x|)| +1 \Bigl), \ \forall x \in D.\]
Indeed, we can write $|x| |y-x^*| = \Bigl|y |x| - \frac{x}{|x|} \Bigl|$ and we deduce putting $\r:= 1- |x|$ that:
\begin{itemize}
\item for $y\in B(x,4\r)\cap D$, $ \Bigl|y |x| - \frac{x}{|x|} \Bigl| \geq  \Bigl| \frac{x}{|x|} \Bigl| - |y| |x| \geq 1-|x|=\r$;
\item for $y\in B(x,4\r)^c\cap D$, $ \Bigl|y |x| - \frac{x}{|x|} \Bigl|^2 - |y-x|^2 = (1-|y|^2)(1-|x|^2) \geq  0$.
\end{itemize}
Using this two inequality, we follow exactly the proof of Lemma  \ref{technic}, which allows us to establish Lemma  \ref{dtL1 est} in the bounded case.
\end{remark}

In light of Lemmas \ref{L1 est} and \ref{dtL1 est}, we see that we have an additional  logarithm which implies that  $\frac{\pd_t L_1}{L_1}\to \infty$ if $x\to \Cc$. However, the logarithm is exactly what we can estimate by Gronwall inequality: $L'(t)=  \frac{\pd_t L_1}{L_1} \approx \ln L_1 = L(t)$. It is the general idea to establish the main result of this section.

\begin{proposition}\label{prop support}
We assume that $\om_0$ is non-positive, compactly supported in $\OM$ and $\g_0\geq - \int \om_0$. Then, for any ${\mathbf{T}^*}>0$, there exists $C_{\mathbf{T}^*}$ such that
\[ L(t) \leq C_{\mathbf{T}^*},\ \forall x_0\in \supp \om_0, \ \forall t\in [0,{\mathbf{T}^*}].\]
\end{proposition}

\begin{proof}
As the support of $\om_0$ does not intersect $\pd \OM$, we have by continuity of $\Tc$ and by Lemma \ref{L1 est} that
\[ L(0) = - \ln L_1(0,x_0) \leq -\ln C_2(|\Tc(x_0)|-1) \]
is bounded uniformly in $x_0 \in \supp \om_0$.

For any, $x_0 \in \supp \om_0$, \eqref{support} gives that $\phi(t)\in B(0,R_{{\mathbf{T}^*}})$, for all $t\in [0,{\mathbf{T}^*}]$. Therefore, the computation made in the begin of this subsection gives
\[L'(t) = -\pd_t L_1(t,\phi (t))  / L_1(t,\phi (t)).\]

As $L_1$ is positive, we have
\[L'(t) = -\pd_t L_1(t,\phi (t))  / L_1(t,\phi (t)) \leq |\pd_t L_1(t,\phi (t))|  / L_1(t,\phi (t)).\]

Lemma \ref{L1 est} states that there exists $C_2$ such that
\begin{equation}\label{L 1}
 L_1(t,\phi (t)) \geq C_2 ( |\Tc(\phi (t))| -1).
\end{equation}
Moreover, thanks to Lemma \ref{L1 maj}, it is easy to find $C_4$ such that
\begin{equation}\label{L 2}
L_1(t,x) \leq C_4, \ \forall x \in B(0,R_{\mathbf{T}^*}\cap \OM), \ \forall t\in [0,{\mathbf{T}^*}].
\end{equation}
Finally, we proved in Lemma \ref{dtL1 est} that there exists $C_3$ such that
\begin{equation}\label{L 3}
|\pd_t L_1(t,\phi (t)) | \leq C_3(|\Tc(\phi (t))|-1)\Bigl(1+ |\ln(|\Tc(\phi (t))|-1)|\Bigl).
\end{equation}

We can easily check that in the interval $(0,\mathrm{e}^{-1})$ the function $x\mapsto x |\ln x|$ is equal to the map $x\mapsto -x \ln x$, which is increasing. By \eqref{L 1} and \eqref{L 2}, we use the fact that
\[ 0\leq \frac{C_2 ( |\Tc(\phi (t))| -1)}{\mathrm{e} C_4} \leq \frac{L_1(t,\phi (t))}{\mathrm{e} C_4} \leq \mathrm{e}^{-1}\]
to apply this remark on \eqref{L 3}:
\begin{eqnarray*}
|\pd_t L_1(t,\phi (t)) |& \leq& C_3(|\Tc(\phi (t))|-1)\Bigl(1+ | \ln \frac{\mathrm{e} C_4}{C_2}|+  |\ln\frac{C_2(|\Tc(\phi (t))|-1)}{\mathrm{e} C_4}|\Bigl)  \\
&\leq& C_3(|\Tc(\phi (t))|-1)\Bigl(1+ | \ln \frac{\mathrm{e} C_4}{C_2}| \Bigl) - \frac{\mathrm{e}C_3C_4}{C_2} \frac{C_2 ( |\Tc(\phi (t))| -1)}{\mathrm{e} C_4} \ln\frac{C_2(|\Tc(\phi (t))|-1)}{\mathrm{e} C_4} \\
&\leq& \frac{C_3}{C_2}\Bigl(1+ | \ln \frac{\mathrm{e} C_4}{C_2}| \Bigl) L_1(t,\phi (t)) - \frac{\mathrm{e}C_3C_4}{C_2}\frac{L_1(t,\phi (t))}{\mathrm{e} C_4}\ln\frac{L_1(t,\phi (t))}{\mathrm{e} C_4} \\
&\leq& L_1(t,\phi (t)) (C_5 - C_6 \ln L_1(t,\phi (t))).
\end{eqnarray*}
As $L_1$ is positive, we finally obtain that
\[L'(t) = \frac{-\pd_t L_1(t,\phi (t))}{L_1(t,\phi (t))} \leq \frac{|\pd_t L_1(t,\phi (t))|}{ L_1(t,\phi (t))} \leq C_5 - C_6 \ln L_1(t,\phi (t)) = C_5 + C_6 L(t).\]
The constants $C_5$ and $C_6$ are uniform for $x_0 \in \supp \om_0$ and $t\in [0,{\mathbf{T}^*}]$. Gronwall's lemma gives us that
\[ L(t) \leq (L(0) + \frac{C_5}{C_6}) \mathrm{e}^{C_6 {\mathbf{T}^*}},\ \forall x_0\in \supp \om_0, \ \forall t\in [0,{\mathbf{T}^*}].\]
 \end{proof}

By Corollary \ref{lem : sign L1}, the corollary of this proposition is that the support of $\om(t,\cdot)$ never meets the boundary. As before, we have the same proposition with opposite sign conditions:
\begin{remark}\label{rem sign}
We assume that the support of $\om_0$ is outside a neighborhood of $\pd \OM$, that $\om_0$ is non-negative and $\g_0\leq - \int \om_0$. Then, for any ${\mathbf{T}^*}>0$, there exists $C_{\mathbf{T}^*}$ such that
\[ L(t) \leq C_{\mathbf{T}^*},\ \forall x_0\in \supp \om_0, \ \forall t\in [0,{\mathbf{T}^*}].\]
\end{remark}
Indeed, replacing everywhere $L_1$ by $-L_1$, the last inequality in the proof would be
\[L'(t) = \frac{\pd_t L_1(t,\phi (t))}{-L_1(t,\phi (t))} \leq \frac{|\pd_t L_1(t,\phi (t))|}{- L_1(t,\phi (t))} \leq C_5 - C_6 \ln -L_1(t,\phi (t)) = C_5 + C_6 L(t),\]
which allows us to conclude in the same way.

\section{Vorticity far from the boundary}\label{sect : 4}

The role of this section is to apply rigorously the idea of the previous section. In Section \ref{sect : 3}, we assume that the flows exist and are regular enough to compute derivatives. However, the solution considered in Theorems \ref{main 1} and \ref{main 2} are weak, and such a property is not established in the existence proofs (see \cite{lac_euler, GV_lac}).

Without considering trajectories, we have proved, thanks to renormalized solutions, that the weak solutions verified the classical estimates:
\begin{itemize}
\item conservation of the total mass of the vorticity \eqref{om-est-1};
\item conservation of the $L^p$ norm of the vorticity for $p\in [1,\infty]$ \eqref{om-est-2};
\item conservation of the circulation \eqref{gamma cons} (only for exterior domains);
\item compact support for the vorticity: Proposition \ref{compact_vorticity} (only for exterior domains).
\end{itemize}

We can easily prove that the conservations of the total mass and the $L^1$ norm of the vorticity  implies that
\begin{equation*}
\om_0 \geq 0\ \text{ a.e. in }\ \OM \Longrightarrow   \om(t,x) \geq 0,\ \forall t\geq 0, \text{ a.e. in }\ \OM.
\end{equation*}

Thinking of the Liapounov function used in Section \ref{sect : 3}, we can construct a good test function in order to use the renormalization theory. We establish now the key result for proving the uniqueness.

\begin{proposition}
\label{constant_vorticity_2} Let $\om$ be a global weak solution of \eqref{transport*} such that $\om_0$ is compactly supported in $\OM$.
If $\om_0$ is non-positive and $\g_0\geq -\int \om_0$ (only for exterior domains), then, for any ${\mathbf{T}^*}>0$, there exists a neighborhood $U_{{\mathbf{T}^*}}$ of $\pd{\OM}$ such that
\begin{eqnarray*}
\om(t)\equiv 0 \qquad \textrm{on \; \;} U_{{\mathbf{T}^*}},\qquad \forall t\in [0,{\mathbf{T}^*}].
\end{eqnarray*}
\end{proposition}

\begin{proof} 
According to Proposition \ref{compact_vorticity}, we have
\begin{equation}
\label{support 2} \supp \om(t) \subset
B\left(0,R_0+ C_0 t)\right), \qquad \forall t\geq 0.
\end{equation}
We note $R_{\mathbf{T}^*}:= R_0+C_0 {\mathbf{T}^*}$. 

Thanks to Lemma \ref{L1 maj}, it is easy to find $C_4$ such that
\begin{equation}\label{L 22}
L_1(t,x) \leq C_4, \ \forall x \in B(0,R_{\mathbf{T}^*}\cap \OM), \ \forall t\in [0,{\mathbf{T}^*}].
\end{equation}
We also deduce from the conservation of the vorticity sign that Corollary \ref{lem : sign L1} holds true.

 We aim to apply \eqref{renorm} with the choice $\b(t)=t^2$ and we set
 \begin{equation*}
\F(t,x)=\chi_0 \left( \frac{-\ln L_1(t,x) + \ln C_4}{R(t)}\right),
\end{equation*}
where $\chi_0$ is a smooth function: $\mathbb{R} \to\R^+$ which is identically zero for $|x|\leq 1/2$ and identically one for $|x|\geq 1$ and increasing on $\R^+$, $L_1$ is defined in \eqref{L1} and $R(t)$ is an increasing continuous function  to be determined later on. 

As $L_1(t,x) \leq C_4$, we have that $-\ln L_1(t,x) + \ln C_4$ is positive $\forall x \in B(0,R_{\mathbf{T}^*}\cap \OM), \ \forall t\in [0,{\mathbf{T}^*}]$.

On one hand, Lemma \ref{L1 est} states that there exists $C_2$ such that
\begin{equation}\label{L 12}
 L_1(t,x) \geq C_2 ( |\Tc(x)| -1), \ \forall x \in B(0,R_{\mathbf{T}^*}\cap \OM), \ \forall t\in [0,{\mathbf{T}^*}].
\end{equation}
Finally, we proved in Lemma \ref{dtL1 est} that there exists $C_3$ such that
\begin{equation}\label{L 32}
|\pd_t L_1(t,x) | \leq C_3(|\Tc(x)|-1)\Bigl(1+ |\ln(|\Tc(x)|-1)|\Bigl), \ \forall x \in B(0,R_{\mathbf{T}^*}\cap \OM), \ \forall t\in [0,{\mathbf{T}^*}].
\end{equation}
Then, using the fact that $x\mapsto -x \ln x$ is increasing in $[0,\mathrm{e}^{-1}]$ (see the proof of Proposition \ref{prop support}) we have that
\begin{equation}\label{L bis}
 |\pd_t L_1(t,x) | \leq L_1(t,x) (C_5-C_6 \ln \frac{L_1(t,x)}{C_4}), \ \forall x \in B(0,R_{\mathbf{T}^*}\cap \OM), \ \forall t\in [0,{\mathbf{T}^*}].
\end{equation}

On the other hand, we have
\begin{equation*}
\nabla_x L_1(t,x)=-u^{\bot}(t,x),
\end{equation*}
therefore
\begin{equation*}
u\cdot \nabla \F=u\cdot u^{\bot}\frac{\chi_0'}{R L_1}\equiv 0.
\end{equation*}
Besides,
\begin{equation*}
\dt \F(t,x)=\Big(\frac{R'(t)}{R^2(t)} \ln \frac{L_1(t,x)}{C_4}-\frac{1}{R} \frac{\pd_t L_1(t,x)}{L_1(t,x)}\Big)\chi_0'\left(\frac{-\ln L_1(t,x) + \ln C_4}{R(t)}\right).
\end{equation*}

In view of \eqref{renorm}, this yields for any\footnote{see the proof of Proposition \ref{compact_vorticity} to check that this equality holds for all $T$.} $T\in [0,{\mathbf{T}^*}]$
\begin{equation*}
\begin{split}
\int_{\RR} &\F(T,x) \om^2(T,x)\,dx -\int_{\RR} \F(0,x)
\om^2_0(x)\,dx\\& =\int_0^T \int_{\RR} \om^2(t,x)
\frac{\chi_0'\left(\frac{-\ln L_1(t,x) + \ln C_4}{R}\right)}{R}\left(\frac{R'}{R} \ln \frac{L_1(t,x)}{C_4}-\frac{\pd_t L_1(t,x)}{L_1(t,x)}\right)\, dx\, dt.
\end{split}
\end{equation*}

Since $-\ln \frac{L_1(t,x)}{C_4} \geq 0$, the term $\chi_0'(\frac{-\ln L_1(t,x) + \ln C_4}{R})$ is non negative and non zero provided  $\frac{1}{2}\leq \frac{-\ln (L_1(t,x)/C_4)}{R}\leq 1$, so we obtain
\begin{equation*}
\begin{split}
\int_{\RR} \F(T,x) \om^2(T,x)\,dx -\int_{\RR} \F(0,x)
\om^2_0(x)\,dx& \leq \int_0^T \int_{\RR} \om^2
\frac{\chi_0'}{R}\left(-\frac{R'}{2}  +C_5 + C_6 R \right)\, dx\, dt.
\end{split}
\end{equation*}
In the last inequality, we have used \eqref{L bis}, which is allowed because $\supp \om \subset B(0,R_{\mathbf{T}^*}\cap \OM)$ for all $t\in [0,{\mathbf{T}^*}]$.

We now choose
\begin{equation*}
R(t)=\l_0 \mathrm{e}^{2C_6 t} -\frac{C_5}{C_6},
\end{equation*}
with $\l_0$ to be determined later on, so that
\begin{equation*}
\int_{\RR} \F(T,x) \om^2(T,x)\,dx \leq \int_{\RR} \F(0,x)
\om^2_0(x)\,dx.
\end{equation*}

Since the support of $\om_0$ does not intersect some neighborhood of $\Cc$, the continuity of $\Tc$ implies that there exists $\m_0>0$ such that $\Tc(\supp \om_0)\subset B(0,\mu_0+1)^c$. Then, 
\[0\leq -\ln L_1(0,x) + \ln C_4 \leq -\ln \Bigl(C_2 ( |\Tc(x)| -1)\Bigl) + \ln C_4 \leq -\ln (C_2\m_0) + \ln C_4\]
for all $x$ in the support of $\om_0$.  We finally choose $\l_0$ so that
 \begin{equation*}
 0<\frac{-\ln (C_2\m_0) + \ln C_4}{\l_0  -\frac{C_5}{C_6}} \leq \frac{1}{2}.
 \end{equation*}
 For this choice, we have
 \begin{equation*}
 \F(0,x)\om_0^2(x)=\chi_0\left(\frac{-\ln L_1(0,x) + \ln C_4}{\l_0  -\frac{C_5}{C_6}}\right)\om_0^2(x)\equiv 0.
 \end{equation*}
 We deduce that for all $T\in [0,{\mathbf{T}^*}]$, $\F(T,x)\om^2(T,x)\equiv 0$. Thanks to Lemma \ref{L1 maj}, we know that there exists $C_1$ such that
 \[ L_1(T,x) \leq C_1(|\Tc(x)|-1)^{1/2}, \ \forall x\in B(0,R_{\mathbf{T}^*}), \ \forall T\in [0,{\mathbf{T}^*}].\]
 Therefore, for any $x\in \Tc^{-1}\Bigl(B(0,1+\mathrm{e}^{-\frac2{C_1} (R({\mathbf{T}^*})-\ln C_4)})\setminus B(0,1)\Bigl)$ and any $T\in[0,{\mathbf{T}^*}]$, we have that
 \begin{eqnarray*}
|\Tc(x)|&\leq& 1+\mathrm{e}^{-\frac2{C_1} (R({\mathbf{T}^*})-\ln C_4)}\\
\ln (|\Tc(x)| -1) &\leq & -\frac2{C_1} (R({\mathbf{T}^*})-\ln C_4)\\
-\frac{C_1}2\ln (|\Tc(x)| -1) &\geq &  (R({\mathbf{T}^*})-\ln C_4)
\end{eqnarray*}
which implies that
\begin{equation}\label{ineq1}
\frac{-\frac{C_1}2\ln (|\Tc(x)| -1) + \ln C_4}{R({\mathbf{T}^*})}\geq 1.
\end{equation}
Moreover, for any $x\in B(0,R_{\mathbf{T}^*})$ and $T\in [0,{\mathbf{T}^*}]$ we have that
\begin{eqnarray*}
\ln L_1(T,x)&\leq & \frac{C_1}2 \ln (|\Tc(x)| -1)\\
- \ln L_1(T,x)+\ln C_4 &\geq & -\frac{C_1}2 \ln (|\Tc(x)| -1)+\ln C_4
\end{eqnarray*}
which gives (using that $R$ is an increasing function and that $- \ln L_1(T,x)+\ln C_4\geq 0$):
\[   \frac{- \ln L_1(T,x)+\ln C_4}{R(T)} \geq  \frac{- \ln L_1(T,x)+\ln C_4}{R({\mathbf{T}^*})} \geq  \frac{-\frac{C_1}2\ln (|\Tc(x)| -1) + \ln C_4}{R({\mathbf{T}^*})}.\]
Putting together the last inequality and \eqref{ineq1}, $\F(T,x)\om^2(T,x)\equiv 0$ for any $T\in [0,{\mathbf{T}^*}]$ implies that
\[ \om(T,x)\equiv 0,\ \forall x\in \Tc^{-1}\Bigl(B(0,1+\mathrm{e}^{-\frac2{C_1} (R({\mathbf{T}^*})-\ln C_4)})\setminus B(0,1)\Bigl), \ \forall T\in [0,{\mathbf{T}^*}]\]
and the conclusion follows.
\end{proof}

\begin{remark} \label{rem : sign}
Of course, as in Remarks \ref{rem : L1 est} and \ref{rem sign}, the previous proposition holds true for the opposite sign condition:

\centerline{$\om_0$  non negative and  $\g_0 \leq -\int \om_0$.}
\end{remark}

Actually, we can prove Propositions \ref{compact_vorticity} and \ref{constant_vorticity_2} without the renormalized solutions. Indeed, as we proved in Remark \ref{remark : conserv} that $\om$ stays definite sign (thanks to the renormalization theory), then we can use $\om$ instead of $\om^2$ in the proofs. In this case, we just need that $\om$ is a weak solution in the sense of distribution. However, we have presented here the proofs with $\b(\om)=\om^2$ in order to extend the theorems in the case where $\om_0$ is constant near the boundary (see Section \ref{sect : 6}).

\section{Uniqueness of Eulerian solutions}\label{sect : 5}

\subsection{Velocity formulation}\ 

In order to follow the proof of Yudovich, we give a velocity formulation\footnote{The original proof comes from \cite{ift_lop_euler} and we copy it for a sake of clarity.} of the extension $\bar u$. 

We begin by introducing 
\[ v(x):=\int_{\R^2} K_{\R^{2}}(x-y) \bar\om(y)dy \] 
with $K_{\R^2}(x)=\frac{1}{2\pi}\frac{x^\perp}{|x|^2}$, the solution in the full plane of
\begin{equation*}
\diver v =0 \text{ on } \R^2, \quad \curl v =\bar\om \text{ on } \R^2, \quad \lim_{|x|\to\infty}|v|=0.
\end{equation*}
This velocity is bounded, and we denote the perturbation by $w=\bar u-v$, which belongs to $L^\infty_{\loc}(\R^+;L^p_{\loc}(\R^2))$ for $p<4$, and verifying
\begin{equation*}
\diver w =0 \text{ on } \R^2, \quad \curl w =g_{\om,\g}(s) \d_{\pd \OM} \text{ on } \R^2, \quad \lim_{|x|\to\infty}|w|=0.
\end{equation*}

We infer that $v$ verifies the following equation:

\begin{equation}
\label{vit_equa}
\begin{cases}
v_t+v\cdot \na v+ v\cdot \na w + w\cdot \na v - v(s)^\perp\tilde g_{v,\g}(s)\cdot \d_{\pd \OM}=-\na p, & \text{ in }\R^2\times(0,\infty)  \\
\diver v=0, & \text{ in }\R^2\times(0,\infty) \\
w(x)=\frac{1}{2\pi} \oint_{\pd \OM} \frac{(x-s)^\perp}{|x-s|^2}\tilde g_{v,\g}(s) { ds}, & \text{ in }\R^2\times(0,\infty)  \\
v(x,0)=K_{\R^2}[\bar \om_0], & \text{ in }\R^2.
\end{cases}
\end{equation}
with $\tilde g_{v,\g}:=g_{\curl v,\g}$ (see \eqref{g_o_bis}).

In order to prove the equivalence of (\ref{tour_equa}) and (\ref{vit_equa}) it is sufficient to show that
\begin{equation}
\label{equiv}
\curl[v \cdot \na w + w \cdot \na v - v(s)^\perp\tilde g_{v,\g}(s)\cdot \d_{\pd \OM}]=\diver(\bar \om w)
\end{equation}
for all divergence free fields $v\in W^{1,p}_{\loc}$, with some $p>2$. Indeed, if (\ref{equiv}) holds, then we get for $\bar \om=\curl v$ 
\begin{eqnarray*}
0 &=& -\curl \na p=\curl[v_t+ v \cdot \na v + v \cdot \na w + w \cdot \na v - v(s)^\perp\tilde g_{v,\g}(s)\cdot \d_{\pd \OM}] \\
&=& \pd_t \bar \om + v \cdot \na \bar \om + w \cdot \na \bar \om = \pd_t \bar \om + \bar u \cdot \na \bar \om=0
\end{eqnarray*}
so relation (\ref{tour_equa}) holds true. And vice versa, if (\ref{tour_equa}) holds then we deduce that the left hand side of (\ref{vit_equa}) has zero curl so it must be a gradient.

We now prove (\ref{equiv}). As $ W^{1,p}_{\loc} \subset \mathcal{C}^0$, $v(s)$ is well defined. Next, it suffices to prove the equality for smooth $v$, since we can pass to the limit on a subsequence of smooth approximations of $v$ which converges strongly in $W^{1,p}_{\loc}$ and $\mathcal{C}^0$. Now, it is trivial to check that, for a $2\times 2$ matrix $A$ with distribution coefficients, we have
\[ \curl \diver A =\diver \begin{pmatrix} \curl C_1 \\ \curl C_2 \end{pmatrix} \]
where $C_i$ denotes the $i$-th column of $A$.
For smooth $v$, we deduce
\begin{eqnarray*}
\curl[v \cdot \na w + w \cdot \na v] &=& \curl \diver(v\otimes w+w \otimes v)\\
&=& \diver\begin{pmatrix} \curl(vw_1)+\curl(wv_1) \\ \curl(vw_2)+\curl(wv_2) \end{pmatrix} \\
&=& \diver(w\ \curl v+v \cdot \na^\perp w+ v\ \curl w+w \cdot \na^\perp v).
\end{eqnarray*}
It is a simple computation to check that
\[ \diver(v \cdot \na^\perp w+w \cdot \na^\perp v) = v \cdot \na^\perp \diver w+ w \cdot \na^\perp \diver v + \curl v\ \diver w+ \curl w\ \diver v. \]
Taking into account that we have free divergence fields, we can finish by writing
\begin{equation*}
\curl[v \cdot \na w + w \cdot \na v] = \diver(w\ \curl v +v \tilde g_{v,\g}(s)\cdot \d_{\pd \OM}) 
= \diver(w\ \curl v) + \curl[v(s)^\perp \tilde g_{v,\g}(s)\cdot \d_{\pd \OM}].
\end{equation*}
which proves (\ref{equiv}).

\subsection{Proof of Theorems \ref{main 1}-\ref{main 2}}\ 

The goal is to adapt the proof of Yudovich: let $u_1$ and $u_2$ be two weak solutions of  \eqref{Euler} (Theorem \ref{theorem1}) from the same initial data $u_0$ verifying \eqref{typeinitialdata}-\eqref{imperm}. We define as above $v_1$, $w_1$ (resp. $v_2$, $w_2$) associated to $\om_1:= \curl u_1$ (resp. $\om_2:= \curl u_2$) and $\g_0$ (see \eqref{g_0} and \eqref{gamma cons}). We denote
\[ \tilde \om := \bar \om_1 - \bar \om_2 \]
where the bar means that we extend by zero outside $\OM$ and
\[ \vti := v_1-v_2,\]
which verifies
\begin{equation}
\label{diff_velocity}
\begin{split}
\pd_t \tilde v + \tilde v\cdot \na v_1+v_2\cdot \na\tilde v + \diver (\tilde v \otimes w_1+& v_2\otimes \tilde w+w_1\otimes \tilde v+\tilde w\otimes v_2)\\ &- (v_1(s)^\perp \tilde g_{\tilde v,0}(s)  - \tilde v(s)^\perp  \tilde g_{v_2,\g_0}(s) )\cdot \delta_{\pd \OM} =-\na \tilde p.
\end{split}
\end{equation}
Next, we will multiply by $\tilde v$ and integrate. The difficulty compared with the Yudovich's original proof is that we have some terms as $\int_{\R^2} |w_1| |\tilde v| |\na \tilde v|$ with $w_1$ blowing up near the corners. The general idea is to divide such an integral in two parts: on $U$ a small neighborhood of the boundary where the vorticity vanishes (see Proposition \ref{constant_vorticity_2}) and on $\R^2 \setminus U$ where the velocity $w_1$ is regular. Far from the boundary, we follow what Yudovich did, and near the boundary we compute
\[ \int_U |w_1| |\tilde v| |\na \tilde v| \leq \| w_1\|_{L^1(U)} \| \tilde v \|_{L^\infty(U)} \| \na \tilde v \|_{L^\infty(U)}.\]
Indeed $w_1$ is integrable near the boundary, and as $\tilde v$ is harmonic in $U$ ($\diver \tilde v = \curl \tilde v =0$), then we have
\[ \int_U |w_1| |\tilde v| |\na \tilde v| \leq C \|\tilde v \|_{L^2(U)}^2\]
which will allow us to conclude by the Gronwall's lemma. We see here why Proposition \ref{constant_vorticity_2} is the main key of the uniqueness proof.

This idea was used in \cite{lac_miot} in order to prove the uniqueness of the vortex-wave system, and we follow the same plan.

\medskip

We denote by $W^{1,4}_\s(\RR)$ the set of functions belonging to $W^{1,4}(\RR)$ and which are divergence-free in the sense of distributions, and by $W^{-1,4/3}_\s(\RR)$
its dual space.

First, we prove that we can multiply by $\tilde v$ and integrate. As a consequence of \eqref{vit_equa} and \eqref{diff_velocity}, we obtain the following properties for $\vti$.

\begin{proposition}
\label{prop : cont-velocity} Let $u_0$ verifying \eqref{typeinitialdata}, $u_1,u_2$ be two weak solutions of \eqref{Euler} with initial condition $u_0$. Let $\vti=v_1-v_2$. Then we have
\begin{equation*}
\vti \in L_{\loc}^2\left(\R^+,W^{1,4}_\s(\RR)\right),\quad \pd_t \vti \in L_{\loc}^2\left(\R^+,W^{-1,\frac{4}{3}}_\s(\RR)\right).
\end{equation*}
In addition, we have $\vti\in  C\left(\R^+, L^2(\RR)\right)$ and for all $T\in\R^+$,
\begin{equation*}
\|\tilde v(T)\|_{L^2(\RR)}^2=2\int_0^T \langle \pd_t \tilde v,\tilde v \rangle_{W^{-1,4/3}_\s,W^{1,4}_\s}\,ds,\qquad \forall T\in \R^+.
\end{equation*}
\end{proposition}
The proof follows easily from the estimates established in Section \ref{sect : 2}. The reader can find the details in Section \ref{sect : technical}.

Now, we take advantage of the fact that $\om_i$ is equal to zero near $\pd \OM$ (Proposition \ref{constant_vorticity_2}) to give harmonic regularity estimates on $\tilde v(t)$.

\begin{lemma} \label{harm} Let $\mathbf{T}^*>0$. We assume that $\om_0$ is compactly supported in $\OM$ and has the sign conditions of Proposition \ref{constant_vorticity_2} (or of Remark \ref{rem : sign}). Then, there exists a neighborhood $U_{\mathbf{T}^*}$ of $\pd{\OM}$ such that for all $t\leq {\T^*}$, $\tilde v(t,\cdot)$ is harmonic on $U_{\T^*}$. In particular, for $O_{\T^*}$ an open set such that $\pd \OM \Subset O_{\T^*} \Subset U_{\T^*}$, we have the following estimates:
\begin{itemize}
\item[(1)] $\|\tilde v(t,.)\|_{L^\infty(O_{\T^*})}\leq C \|\tilde v(t,.)\|_{L^2(\R^2)}$,
\item[(2)] $ \|\na\tilde v(t,.)\|_{L^\infty(O_{\T^*})}\leq C \|\tilde v(t,.)\|_{L^2(\R^2)}$,
\end{itemize}
where $C$ only depends on $O_{\T^*}$.
\end{lemma}

The proof is a direct consequence of the mean-value formula (see e.g. the proof of Lemma 3.9 in \cite{lac_miot}). In order to prepare the Gronwall estimate, we establish the following estimates on $w_1-w_2$.

\begin{lemma} \label{est w} Let $\mathbf{T}^*>0$ and $\pd \OM \Subset O_{\T^*} \Subset U_{\T^*}$ as Lemma \ref{harm}. Then $\tilde w:= w_1-w_2$ verifies the following estimates for any $t\in[0,{\T^*}]$:
\begin{itemize}
\item[(1)] $\|\tilde w(t,\cdot)\|_{L^2(\R^2)} \leq 2 \|\tilde v(t,\cdot)\|_{L^2(\R^2)}$,
\item[(2)] $ \|\tilde w(t,\cdot)\|_{L^\infty(O_{\T^*}^c)}\leq C \|\tilde v(t,\cdot)\|_{L^2(\R^2)}$,
\item[(3)] $ \|\na\tilde w(t,\cdot)\|_{L^2(O_{\T^*}^c)}\leq C \|\tilde v(t,\cdot)\|_{L^2(\R^2)}$,
\end{itemize}
where $C$ only depends on $O_{\T^*}$.
\end{lemma}

\begin{proof}
We fix $t\in [0,{\T^*}]$ and we denote $\tilde u:= \bar u_1-\bar u_2$. From the explicit formula and the conservation law, we have that
\begin{equation*}
\left\lbrace\begin{aligned}
\diver \tilde u &=0 &\text{ on } \OM, \\
\curl \tilde u &= \tilde \om &\text{ on } \OM, \\
\tilde u\cdot \hat n &=0 &\text{ on } \pd \OM, \\
\int_{\pd \OM} \tilde u\cdot \hat \t &=0 &\text{ (only if $\OM$ is an exterior domain)},\\
\lim_{|x|\to\infty}|\tilde u|&=0 &\text{ (only if $\OM$ is an exterior domain)},
\end{aligned}\right.
\end{equation*}
and
\begin{equation*}
\left\lbrace\begin{aligned}
\diver \tilde v &=0 &\text{ on } \OM, \\
\curl \tilde v &= \tilde\om &\text{ on } \OM, \\
\int_{\pd \OM} \tilde v\cdot \hat \t &=0 &\text{ (only if $\OM$ is an exterior domain)},\\
\lim_{|x|\to\infty}|\tilde v|&=0 &\text{ (only if $\OM$ is an exterior domain)}.
\end{aligned}\right.
\end{equation*}
Indeed, in the case of exterior domains, $\tilde \om\equiv 0$ on $\Cc$ which implies that the circulation of $\tilde v$ around $\Cc$ is equal to zero. Therefore, we have the following.

\begin{lemma} $\tilde u$ is the orthogonal projection of $\tilde v$ on the set of the vector field defined on $\OM$ square integrable, divergence free and tangent to the boundary. Therefore we have:
\[ \| \tilde u(t,\cdot) \|_{L^2(\OM)} \leq \| \tilde v(t,\cdot) \|_{L^2(\OM)} .\]
\end{lemma}
This lemma is a classical property of the Leray projector in arbitrary domains (see \cite[Theo 1.1 in Chap III.1.]{Galdi}). Then the first point is a direct consequence of this lemma:
\[ \|\tilde w(t,\cdot)\|_{L^2(\R^2)} \leq \|\tilde u(t,\cdot)\|_{L^2(\OM)} + \|\tilde v(t,\cdot)\|_{L^2(\R^2)} \leq  \|\tilde v(t,\cdot)\|_{L^2(\OM)} + \|\tilde v(t,\cdot)\|_{L^2(\R^2)} \leq 2 \|\tilde v(t,\cdot)\|_{L^2(\R^2)}.\]

The second point is exactly the same thing as in Lemma \ref{harm}: $\tilde w$ is harmonic in $\OM$ then there exists $C$ depending on $O_{\T^*}$ such that
\[ \|\tilde w(t,\cdot)\|_{L^\infty(O_{\T^*}^c)} \leq C \|\tilde w(t,\cdot)\|_{L^2(\OM)} \leq 2C \|\tilde v(t,\cdot)\|_{L^2(\R^2)}.\]
Another consequence of the mean-value Theorem is that
\[ \|\na \tilde w(t,\cdot)\|_{L^2(O_{\T^*}^c)} \leq C \|\tilde w(t,\cdot)\|_{L^2(\OM)} \leq 2C \|\tilde v(t,\cdot)\|_{L^2(\R^2)}.\]
Indeed, there is $R_1$ such that dist$(\pd \OM,\pd O_{\T^*})>R_1$, then
\begin{eqnarray*}
\|\na \tilde w(t,x)\|_{L^2(O_{\T^*}^c)} &=& \Bigl\| \frac{1}{\pi R_1^2} \int_{B(x,R_1)} \na \tilde w(t,y)\, dy \Bigl\|_{L^2(O_{\T^*}^c)}=\Bigl\| \frac{1}{\pi R_1^2} \int_0^{2\pi} \tilde w(t,x+R_1e^{i\th}) \n \, R_1 d\th \Bigl\|_{L^2(O_{\T^*}^c)}\\
&\leq& \int_0^{2\pi}\frac{1}{\pi R_1} \|  \tilde w(t,x+R_1e^{i\th}) \|_{L^2(O_{\T^*}^c)} \, d\th \leq \frac{2 \|\tilde w(t,\cdot)\|_{L^2(\OM)}}{R_1}.
\end{eqnarray*}
\end{proof}
\begin{remark}We remark that the result from Galdi's book does not require regularity of $\pd \OM$ when we consider the $L^2$ norm (thanks to the Hilbert structure). In contrast for $p\neq 2$, he states that the Leray projector is continuous from $L^p$ to $L^p$ if the boundary $\pd \OM$ is $C^2$. Indeed, in our case we see that $\tilde v$ belongs to $L^p$ for any $p>1$, whereas $\tilde u = \mathbb{P} \tilde v$ does not belongs in $L^p(\OM)$ for some $p>4$ (if there is an angle greater than $\pi$, see Remark \ref{DT loc}).
\end{remark}

\medskip

We can adapt now the Yudovich proof, as it is done in \cite{lac_miot}.

We fix ${\T^*}>0$ in order to fix $O_{\T^*}$ in Lemmata \ref{harm} and \ref{est w}. We consider smooth and divergence-free functions $\F_n\in C^\infty_c\left( \R^+\times \R^2\right)$  converging to $\tilde{v}$ in $L^2_{\loc}\left(\R^+,W^{1,4}(\RR)\right)$ as test functions in \eqref{diff_velocity}, and let $n$ goes to $+\infty$. First, we have for all $T\in [0,{\T^*}]$
\begin{equation*}
\int_0^T \langle \pd_t \tilde v,\F_n \rangle_{W^{-1,4/3}_\s,W^{1,4}_\s}\, ds\to \int_0^T \langle \pd_t \tilde v,\tilde{v} \rangle_{W^{-1,4/3}_\s,W^{1,4}_\s}\,ds,
\end{equation*}
and we deduce the limit in the other terms from the several bounds for $v_i$ stated in the proof of Proposition \ref{prop : cont-velocity}. This yields
\begin{equation}
\label{diff}
\frac{1}{2}\| \tilde v (T,\cdot)\|_{L^2}^2 =I+J+K,
\end{equation}
where
\begin{equation*}
 \begin{split}
I&=- \int_0^T \int_{\R^2} \tilde v \cdot(\tilde v\cdot \na v_1+v_2\cdot \na\tilde v)\,dx\,dt,\\
J&= \int_0^T \int_{\R^2} (\tilde v \otimes w_1+ v_2\otimes \tilde w+w_1\otimes \tilde v+\tilde w\otimes v_2):\na\tilde v\,dx\,dt, \\
K&= \int_0^T \int_{\pd \OM} v_1(s)^\perp \tilde g_{\tilde v,0}(s)  \cdot \tilde v(s)\,ds.
\end{split}
\end{equation*}
The goal is to estimate all the terms in the right-hand side in order to obtain a Gronwall-type inequality.

For the first term $I$ in \eqref{diff}, we begin by noticing that
\begin{equation*}\int_{\RR}(v_2\cdot \na\tilde v)\cdot \tilde v\,dx=\frac{1}{2}\int_{\RR} v_{2}\cdot\na |\tilde v|^2\,dx=-\frac{1}{2}\int_{\RR} |\tilde v|^2 \diver v_2\,dx=0,
\end{equation*}
where we have used that $v_2=\Oc(1/|x|)$ and $\tilde v=\Oc(1/|x|^2)$ at infinity. Moreover, H\"older's inequality gives
\begin{equation*}
\left|\int_{\RR}(\tilde v\cdot \na v_1)\cdot \tilde v\,dx\right|\leq\|\tilde v\|_{L^2}\|\tilde v\|_{L^q}\|\na v_1\|_{L^p},
\end{equation*}
with $\frac{1}{p}+\frac{1}{q}=\frac{1}{2}$. On one hand, Calder\'on-Zygmung inequality states that $\|\na v_1\|_{L^p}\leq Cp\|\om_1\|_{L^p}$ for $p\geq 2$. On the other hand, we write by interpolation $\|\tilde v\|_{L^q}\leq \|\tilde v\|^a_{L^2}\|\tilde v\|^{1-a}_{L^\infty}$ with $\frac{1}{q}=\frac{a}{2}+\frac{1-a}{\infty}$. We have that $a=1-\frac{2}{p}$, so we are led to
\begin{equation}\label{eq : I}
|I| \leq Cp\int_0^T \|\tilde v\|_{L^2}^{2-2/p}\,dt.
\end{equation}

We now estimate $J$. We have
\begin{equation*}
\begin{split}
\int_{\RR} (\tilde v\otimes w_1): \na \tilde v\,dx& = \int_{\RR}\sum_{i,j}\tilde v_i w_{1,j}\pd_j\tilde v_i\,dx = \frac{1}{2}\sum_i\int_{\RR}\sum_j w_{1,j}\pd_j\tilde v_i^2\,dx\\
&=-\frac{1}{2}\sum_i\int_{\RR} \tilde v_i^2 \diver w_1\,dx=0,
\end{split}
\end{equation*}
since $w_1$ is divergence-free, and
\begin{equation}
\label{special}
\Bigl| \int_0^T \int_{\R^2} (w_1\otimes \tilde v):\na\tilde v\,dx\,dt \Bigl| \leq \Bigl| \int_0^T \int_{O_{\T^*}} (w_1\otimes \tilde v ):\na \tilde v\,dx\,dt \Bigl| 
+ \Bigl| \int_0^T \int_{O_{\T^*}^c} (w_1\otimes \tilde v ):\na\tilde v\,dx\,dt \Bigl|.
\end{equation}
We perform an integration by part for the second term in the right-hand side of \eqref{special}. Arguing that $\diver \vti=0$, we obtain
\begin{eqnarray*}
\Bigl| \int_0^T \int_{\R^2} (w_1\otimes \tilde v) : \na \tilde v \, dx \, dt \Bigl| &\leq & \Bigl| \int_0^T \int_{O_{\T^*}} (w_1\otimes \tilde v ) : \na \tilde v \, dx \, dt \Bigl|\\
&& + \Bigl| - \int_0^T \Bigl(\int_{O_{\T^*}^c} (\tilde v \cdot  \na w_1 )\cdot \tilde v \, dx + \int_{\pd O_{\T^*}} (w_1\cdot \vti)(\vti\cdot \n) \, ds \Bigl)\, dt \Bigl| \\
&\leq& \int_0^T  \|w_1\|_{L^1(O_{\T^*})} \|\tilde v\|_{L^\infty(O_{\T^*})} \|\na \tilde v\|_{L^\infty(O_{\T^*})} \, dt \\
&& + \int_0^T \|\na w_1\|_{L^\infty(O_{\T^*}^c)} \|\tilde v\|^2_{L^2} \, dt\\
&& + \int_0^T \| w_1\|_{L^\infty(\pd O_{\T^*})} \|\tilde v\|^2_{L^\infty(\pd O_{\T^*})} | \pd O_{\T^*}| \, dt.
\end{eqnarray*}
As we remarked when we introduce $w$: $\|w_1\|_{L^1(O_{\T^*})}\leq C$ with $C$ depending only on $\OM$, ${\T^*}$ and $u_0$. Moreover, using the harmonicity of $w_1$, we know that $\|\na w_1\|_{L^\infty(O_{\T^*}^c)}$ is bounded by a constant times $\|w_1\|_{L^\infty(V_{\T^*}^c)}$, with $\pd \OM \Subset V_{\T^*} \Subset O_{\T^*}$. Using the behavior of $D\Tc$ at infinity (Proposition \ref{T-inf}), Proposition \ref{biot est}, conservation laws \eqref{om-est-1}, \eqref{om-est-2}, \eqref{gamma cons}, then \eqref{biot unbd} allows us to state that $\| u_1 \|_{L^\infty((0,{\T^*})\times V_{\T^*}^c)}\leq C_0$ with $C_0$ depending only on $\OM$, ${\T^*}$ and $u_0$. As $v_1$ is uniformly bounded, we obtain that $\|\na w_1\|_{L^\infty(O_{\T^*}^c)}$ and $\| w_1\|_{L^\infty(\pd O_{\T^*})}$ is bounded uniformly in $(0,{\T^*})$. Then, according to Lemma \ref{harm}, this gives
\begin{equation*}
\Bigl| \int_0^T \int_{\RR} (w_1 \otimes \tilde v):\na \tilde v\,dx\,dt \Bigl| \ \ \leq\ \  C \int_0^T \|\tilde v\|^2_{L^2} \, dt.
\end{equation*}
 In the same way, we obtain by integration by part
\begin{eqnarray*}
\Bigl| \int_0^T \int_{\R^2} (v_2\otimes \tilde w):\na\tilde v \, dx \, dt \Bigl|  
&\leq  &  \Bigl| \int_0^T \int_{O_{\T^*}} (v_2\otimes \tilde w):\na\tilde v\,dx\,dt \Bigl|\\
&& + \Bigl|- \int_0^T \Bigl(\int_{O_{\T^*}^c} (\tilde w \cdot  \na v_2 )\cdot \tilde v\,dx + \int_{\pd O_{\T^*}} (v_2\cdot\vti)(\tilde w \cdot\n) ds \Bigl)\,dt\Bigl|.
\end{eqnarray*}
Therefore,
\begin{eqnarray*}
\Bigl| \int_0^T \int_{\R^2} (v_2\otimes \tilde w):\na\tilde v\,dx\,dt \Bigl| &\leq & 
\int_0^T  \|\tilde w\|_{L^2(O_{\T^*})} \|v_2\|_{L^2(O_{\T^*})} \|\na \tilde v\|_{L^\infty(O_{\T^*})} \, dt \\
&&+ \int_0^T \| \tilde w\|_{L^\infty(O_{\T^*}^c)} \|\tilde v\|_{L^2} \|\na v_2\|_{L^2} \,dt \\
&&+ \int_0^T\| \tilde w\|_{L^\infty(\pd O_{\T^*})} \|\tilde v\|_{L^\infty(\pd O_{\T^*})} \| v_2\|_{L^\infty}|\pd B|\,dt.
\end{eqnarray*}
Using again Calder\'on-Zygmund inequality for $v_2$ and Lemmata \ref{harm} and \ref{est w}, we get
\begin{equation*}
\Bigl| \int_0^T \int_{\R^2} (v_2\otimes \tilde w):\na\tilde v\,dx\,dt \Bigl| \leq C \int_0^T  \|\tilde v\|_{L^2}^2 \, dt.
\end{equation*}
A very similar computation yields
\begin{eqnarray*}
\Bigl| \int_0^T \int_{\R^2} (\tilde w\otimes v_2): \na\tilde v\,dx\,dt \Bigl|  &\leq&
C \int_0^T \left(\|\tilde w\|_{L^2(O_{\T^*})}+\| \na \tilde w\|_{L^2(O_{\T^*}^c)} +\| \tilde w\|_{L^\infty(\pd O_{\T^*})}\right) \|\tilde v\|_{L^2}\, dt\\
&\leq& C \int_0^T  \|\tilde v\|_{L^2}^2 \, dt.
\end{eqnarray*}
Therefore, we arrive at
\begin{equation}\label{eq : J}
|J|\leq 3C \int_0^T  \|\tilde v\|_{L^2}^2\,dt.
\end{equation}

Finally, using \eqref{g_o_bis} we write the third term $K$ in \eqref{diff} as follows:
\begin{eqnarray*}
K &=&\pm \int_0^T \int_{\pd \OM} (\tilde u \cdot \hat \t) (v_1^\perp   \cdot \tilde v)\,ds\\
&=& \pm \int_0^T \int_{\OM}  \curl \tilde u(v_1^\perp   \cdot \tilde v)\, dx \pm \int_0^T \int_{\OM}  \tilde u \cdot \na^\perp (v_1^\perp   \cdot \tilde v)\, dx,
\end{eqnarray*}
where $\pm$ depends if we treat exterior or interior domains.
Using that $\curl \tilde u= \curl \tilde v$ in $\OM$, $\diver \tilde v = 0$ and the behaviors at infinity, we obtain by several integrations by parts:
\begin{equation*}\begin{split}
\int_{ \OM}  \curl \tilde u(v_1^\perp   \cdot \tilde v)\, dx =& \int_{\OM}  \curl \tilde v (v_1^\perp   \cdot \tilde v)\, dx = \int_{\R^2}  \curl \tilde v (v_1^\perp   \cdot \tilde v)\, dx\\
=& \int_{\R^2} \Bigl(-v_{1,2} \tilde v_1\pd_1 \tilde v_2+v_{1,2} \frac{\pd_2 | \tilde v_1|^2}{2}+v_{1,1} \frac{\pd_1 | \tilde v_2|^2}{2}  - v_{1,1}  \tilde v_2 \pd_2 \tilde v_1\Bigl)\, dx\\
= \int_{\R^2} \Bigl(\pd_1 v_{1,2} \tilde v_1 \tilde v_2&+ v_{1,2} \pd_1 \tilde v_1 \tilde v_2 - \pd_2 v_{1,2} \frac{| \tilde v_1|^2}{2}- \pd_1 v_{1,1} \frac{ | \tilde v_2|^2}{2}  + \pd_2 v_{1,1}  \tilde v_2  \tilde v_1+  v_{1,1}  \pd_2\tilde v_2  \tilde v_1\Bigl)\, dx\\
= \int_{\R^2} \Bigl(\pd_1 v_{1,2} \tilde v_1 \tilde v_2&- v_{1,2} \frac{\pd_2 |\tilde v_2|^2}{2} - \pd_2 v_{1,2} \frac{| \tilde v_1|^2}{2}- \pd_1 v_{1,1} \frac{ | \tilde v_2|^2}{2}  + \pd_2 v_{1,1}  \tilde v_2  \tilde v_1 -  v_{1,1} \frac{ \pd_1 |\tilde v_1|^2}{2} \Bigl)\, dx\\
= \int_{\R^2} \Bigl(\pd_1 v_{1,2} \tilde v_1 \tilde v_2&+ \pd_2 v_{1,2} \frac{ |\tilde v_2|^2}{2} - \pd_2 v_{1,2} \frac{| \tilde v_1|^2}{2}- \pd_1 v_{1,1} \frac{ | \tilde v_2|^2}{2}  + \pd_2 v_{1,1}  \tilde v_2  \tilde v_1 + \pd_1 v_{1,1} \frac{ |\tilde v_1|^2}{2} \Bigl)\, dx.
\end{split}\end{equation*}
Hence,
\[ \Bigl| \int_0^T \int_{\OM}  \curl \tilde u(v_1^\perp   \cdot \tilde v)\, dx \, dt\Bigl| \leq  4 \int_0^T \int_{\R^2} |\na v_1| |\tilde v|^2\, dx \, dt\]
which gives by Calder\'on-Zygmund inequality (as for $I$):
\[ \Bigl| \int_0^T \int_{\OM}  \curl \tilde u(v_1^\perp   \cdot \tilde v)\, dx\, dt \Bigl| \leq Cp\int_0^T \|\tilde v\|_{L^2}^{2-2/p}\,dt.\]

With similar computation, and using Lemmata \ref{harm} and \ref{est w}, we can prove that the second term of $K$ can be treated thanks to:
\begin{eqnarray*}
 \int_0^T \int_{\OM}  |\tilde u| |\na v_1| | \tilde v| \, dx\, dt &\leq& Cp\int_0^T \|\tilde v\|_{L^2}^{2-2/p}\,dt \\
 \int_0^T \int_{O_{\T^*}}  |\tilde u| | v_1| |\na \tilde v| \, dx \, dt&\leq& C \int_0^T  \|\tilde v\|_{L^2}^2 \, dt \\
 \int_0^T \int_{O_{\T^*}^c}  |\tilde v| | \na v_1| |\tilde v| \, dx \, dt&\leq& Cp\int_0^T \|\tilde v\|_{L^2}^{2-2/p}\,dt \\
 \int_0^T \int_{O_{\T^*}^c}  |\tilde w| | \na v_1| |\tilde v| \, dx \, dt&\leq& C \int_0^T  \|\tilde v\|_{L^2}^2 \, dt \\
 \int_0^T \int_{O_{\T^*}^c}  |\na \tilde w| | v_1| |\tilde v| \, dx \, dt&\leq&  C \int_0^T  \|\tilde v\|_{L^2}^2 \, dt \\
 \int_0^T \int_{\pd O_{\T^*}} (|\tilde v| + |\tilde w| ) | v_1| |\tilde v| \, dx \, dt&\leq&  C \int_0^T  \|\tilde v\|_{L^2}^2 \, dt,
\end{eqnarray*}
which implies that
\begin{equation}\label{eq : K}
|K|\leq C \int_0^T  \|\tilde v\|_{L^2}^2\,dt+Cp\int_0^T \|\tilde v\|_{L^2}^{2-2/p}\,dt.
\end{equation}

\bigskip

Therefore, the estimates \eqref{eq : I}, \eqref{eq : J} and \eqref{eq : K} with \eqref{diff} establish that
\[ \| \tilde v (T,\cdot)\|_{L^2}^2 \leq C \int_0^T  \|\tilde v\|_{L^2}^2\,dt+Cp\int_0^T \|\tilde v\|_{L^2}^{2-2/p}\,dt.\]
As we choose $p>2$ and as $\|\tilde v\|_{L^2} \leq C_0$ for all $t\in [0,{\T^*}]$ (see Proposition \ref{prop : cont-velocity}), we have $ \|\tilde v\|_{L^2}^{2/p} \leq C_0^{2/p}$ which implies that for $p$ large enough, the previous inequality gives
\[ \| \tilde v (T,\cdot)\|_{L^2}^2 \leq 2Cp\int_0^T \|\tilde v\|_{L^2}^{2-2/p}\,dt.\]

Using a Gronwall-like argument, this implies
\begin{equation*}
\| \tilde v (T,\cdot)\|_{L^2}^2  \leq (2CT)^p,\qquad \forall p\geq 2.
\end{equation*}
Letting $p$ tend to infinity, we conclude that $\| \tilde v (T,\cdot)\|_{L^2} = 0$ for all $T<\min ({\T^*}, 1/(2C))$. Finally, we consider the maximal interval of $[0,{\T^*}]$ on which $\|\tilde v (T,\cdot)\|_{L^2} \equiv 0$, which is closed by continuity of $\| \tilde v (T,\cdot)\|_{L^2}$. If it is not equal to the whole of $[0,{\T^*}]$, we may repeat the proof above, which leads to a contradiction by maximality. Therefore uniqueness holds on $[0,{\T^*}]$, and this concludes the proof of Theorems \ref{main 1} and \ref{main 2}. Indeed, Lemma \ref{est w} implies that $\| u_1-u_2 \|_{L^2} \leq \| \tilde w \|_{L^2} + \| \tilde v \|_{L^2} \leq 2 \| \tilde v \|_{L^2}$. 

\section{Technical results}\label{sect : technical}

We will use several times the following from \cite{ift}:
\begin{lemma}\label{ift}
Let $S\subset\R^2$, $\a\in (0,2)$ and $g:S\to\R^+$ be a function belonging in $L^1(S)\cap L^r(S)$, for $r > \frac{2}{2-\a}$. Then
\[\int_S \frac{g(y)}{|x-y|^\a}dy\leq C\|g\|_{L^1(S)}^{\frac{2-\a-2/r}{2-2/r}}\|g\|_{L^r(S)}^{\frac{\a}{2-2/r}}.\]
\end{lemma}

\subsection{Proof of Proposition \ref{biot est}}

We make the proof in the unbounded case (which is the hardest case).
We decompose $R[\om]$ in two parts:
\[ R_1(x):=  \int_{\OM} \dfrac{(\Tc(x)-\Tc(y))^\perp}{|\Tc(x)-\Tc(y)|^2} \om(y)\, dy \text{ and } R_2(x):=  \int_{\OM} \dfrac{(\Tc(x)- \Tc(y)^*)^\perp}{|\Tc(x)- \Tc(y)^*|^2} \om(y)\, dy.\]

{\it a) Estimate and continuity of $R_1$.} 

Let $z:= \Tc(x)$ and $f(\y):= \om(\Tc^{-1}(\y)) |\det (D\Tc^{-1}(\y)) | \h_{\{|\y|\geq 1\}}$, with $\h_E$ the characteristic function of the set $E$. Making the change of variables $\y= \Tc(y)$, we find
\begin{equation*}
R_1 (\Tc^{-1}(z) ) = \int_{\R^2} \dfrac{(z-\y)^\perp}{|z-\y|}f(\y) \, d\y.
\end{equation*}

Changing variables back, we get 
\[ \|f \|_{L^1(\R^2)}=\|\om\|_{L^1}.\]
We choose $p_0>2$ such that $\det (D\Tc^{-1})$ belongs to $L^{p_0}_{\loc}(\overline{\OM})$ (see Remark \ref{DT loc}). If all the angles are greater than $\pi$, we can choose $p_0=\infty$ (thanks to Theorem \ref{grisvard} and Proposition \ref{T-inf}) and we would have $\|f \|_{L^\infty(\R^2)} \leq C\|\om \|_{L^\infty}.$ However, if there is one angle less than $\pi$, we have to decompose the integral in two parts:
\[R_1 (\Tc^{-1}(z) ) = \int_{|\y|\geq 2} \dfrac{(z-\y)^\perp}{|z-\y|^2}f(\y) \, d\y +  \int_{|\y| \leq 2} \dfrac{(z-\y)^\perp}{|z-\y|^2}f(\y) \, d\y\]
with
\[ \|f \|_{L^\infty(\R^2\setminus B(0,2))} \leq C_1 \|\om \|_{L^\infty}\]
by Proposition \ref{T-inf}, and 
\[ \|f \|_{L^{p_0}(B(0,2))} \leq C_2 \|\om \|_{L^\infty},\]
by Remark \ref{DT loc}. Then we use the classical estimate for the Biot-Savart kernel in $\R^2$ (see Lemma \ref{ift}):
\[ \Bigl|  \int_{|\y|\geq 2} \dfrac{(z-\y)^\perp}{|z-\y|^2}f(\y) \, d\y \Bigl| \leq C_0  \|f \|_{L^1(\R^2\setminus B(0,2))}^{1/2} \|f \|_{L^\infty(\R^2\setminus B(0,2))}^{1/2}  \leq C_4  \| \om \|_{L^1}^{1/2} \| \om \|_{L^\infty}^{1/2}\]
and
\[ \Bigl|  \int_{|\y|\leq 2} \dfrac{(z-\y)^\perp}{|z-\y|^2}f(\y) \, d\y \Bigl| \leq C_0 \|f \|_{L^1(B(0,2))}^{\frac{p_0-2}{2(p_0-1)}} \|f \|_{L^{p_0}(B(0,2))}^{\frac{p_0}{2(p_0-1)}}   \leq C_5 \| \om \|_{L^1}^{\frac{p_0-2}{2(p_0-1)}} \| \om \|_{L^\infty}^{\frac{p_0}{2(p_0-1)}}  \]
which gives the uniform estimate
\[ \|R_1 \|_{L^\infty(\OM)} \leq C (\| \om \|_{L^1}^{1/2} \| \om \|_{L^\infty}^{1/2} + \| \om \|_{L^1}^{a} \| \om \|_{L^\infty}^{1-a} )\]
with $a=\frac{p_0-2}{2(p_0-1)}$ including in $(0,1/2]$. Concerning the continuity, we approximate $f \chi_{B(0,2)}$ by $f_n \in C^\infty_c(B(0,2))$ and $f \chi_{B(0,2)^c}$ by $g_n \in C^\infty_c(B(0,2)^c)$ such that
\[ \|f_n-f\|_{L^1\cap L^{p_0}(B(0,2))} \to 0,\ \|g_n-f\|_{L^1(B(0,2)^c)} \to 0, \ \|g_n\|_{L^\infty} \leq C(f) \text{ as } n\to \infty.\]
As $f_n$ and $g_n$ are smooth, we infer that the functions
\[ z \mapsto  \int_{\R^2} \dfrac{\x^\perp}{|\x|^2}f_n(z-\x) \, d\x \text{ and }  t \mapsto  \int_{\R^2} \dfrac{\x^\perp}{|\x|^2}g_n(z-\x) \, d\x \]
are continuous. Moreover, we deduce from the previous estimates that
\begin{equation*}\begin{split}
\Bigl\| R_1 &(\Tc^{-1}(z)) - \int_{\R^2} \dfrac{(z-\y)^\perp}{|z-\y|^2}g_n(\y) \, d\y- \int_{\R^2} \dfrac{(z-\y)^\perp}{|z-\y|^2}f_n(\y) \, d\y\Bigl\|_{L^\infty(B(0,1)^c)}\\
&\leq C_0  \Bigl(\|f-g_n\|_{L^1(B(0,2)^c)}^{1/2} \|f-g_n \|_{L^\infty(B(0,2)^c)}^{1/2} + \|f-f_n \|_{L^1(B(0,2))}^{\frac{p_0-2}{2(p_0-1)}} \|f-f_n \|_{L^{p_0}(B(0,2))}^{\frac{p_0}{2(p_0-1)}} \Bigl).
\end{split}
\end{equation*}
Thanks to the limit $n\to \infty$, we prove the continuity of $R_1 \circ \Tc^{-1}$. Using Theorem \ref{grisvard}, we conclude that $R_1$ is continuous up to the boundary.

\medskip

{\it b) Estimate and continuity of $R_2$.} 

We use, as before, the notations $f$, $z$ and the change of variables $\y$
\begin{eqnarray*}
R_2 (\Tc^{-1}(z) ) &=&  \int_{|\y|\geq 1}\dfrac{(z-\y^*)^\perp} {|z- \y^*|^2}f(\y)d\y\\
&=&   \int_{|\y|\geq 2}\dfrac{(z-\y^*)^\perp} {|z- \y^*|^2}f(\y)d\y +   \int_{1\leq |\y|\leq 2}\dfrac{(z-\y^*)^\perp} {|z- \y^*|^2}f(\y)d\y \\
&:=& R_{21}(z)+R_{22}(z).
\end{eqnarray*}
If $|\y| \geq 2$, $|z- \y^*|\geq 1/2$ because $|z|\geq 1$ (see the definition of $\Tc$). Therefore, we obtain obviously that
\[ \| R_{21} \|_{L^{\infty}(B(0,1)^c)} \leq 2 \|f\|_{L^{1}(B(0,2)^c)} \leq 2 \| \om \|_{L^1}.\]
The continuity is easier than above:
\begin{itemize}
\item we approximate $f \chi_{B(0,2)^c}$ by $g_n \in C^\infty_c(B(0,2)^c)$ such that $\|g_n-f\|_{L^1(B(0,2)^c)} \to 0$ as $n\to \infty$;
\item the functions
\[ z \mapsto  \int_{|\y|\geq 2}\dfrac{(z-\y^*)^\perp} {|z- \y^*|^2}g_n(\y)d\y \]
is continuous up to the boundary $\pd B(0,1)$ because $|z- \y^*|\geq 1/2$;
\item the previous estimates gives
\[\Bigl\| R_{21}(z)-\int_{|\y|\geq 2}\dfrac{(z-\y^*)^\perp} {|z- \y^*|^2}g_n(\y)d\y \Bigl\|_{L^\infty(B(0,1)^c)} \leq 2 \|f-g_n \|_{L^{1}(B(0,2)^c)};\]
\end{itemize}
which gives the continuity of $R_{21}$.

Concerning $R_{22}$, we again change variables writing $\th=\y^*$, to obtain:
\[ R_{22} (z) =  \int_{1/2 \leq|\th|\leq 1} \frac{(z-\th)^\perp}{|z-\th|^2} f(\th^*) \frac{d\th}{|\th|^4}.\]
Let $g(\th):= \frac{f(\th^*)}{|\th|^4}$. As above, we deduce by changing variables back that
\[ \|g \|_{L^1(1/2\leq|\th|\leq 1)} \leq \|\om \|_{L^1}.\]
It is also easy to see that
\[ \|g\|_{L^{p_0}(1/2\leq|\th|\leq 1)}\leq  2^{\frac{4(p_0-1)}{p_0}} \| f \|_{L^{p_0}(B(0,2))}\leq C_6 \| \om \|_{L^\infty}.\]
Then, by the classical estimates of the Biot-Savart law in $\R^2$, we have
\[ \| R_{22} \|_{L^\infty(B(0,1)^c)} \leq C \| \om \|_{L^1}^{a} \| \om \|_{L^\infty}^{1-a}.\]
Reasoning as for $R_1$, where we approximate $g$, we get that $R_{22}$ is continuous.

The continuity of $\Tc$ allows us to conclude that $R_2$ is continuous up to the boundary, which ends the proof in the case of $\OM$ unbounded.

\medskip

{\bf Remark about the bounded case.}

Concerning $R_1$, we do not need to decompose the integral in two parts:
\[ \|f \|_{L^{p_0}(B(0,1))} \leq C_2 \|\om \|_{L^\infty},\]
where $f(\y):= \om( \Tc^{-1}(\y)) |\det (D\Tc^{-1}(\y)) | \h_{\{|\y|\leq 1\}}$.

Even of $R_2$, we directly have
\[R_2 (\Tc^{-1}(z) ) =  \int_{|\th|\geq 1} \frac{(z-\th)^\perp}{|z-\th|^2} f(\th^*) \frac{d\th}{|\th|^4} \]
and we conclude following the proof concerning $R_{22}$.

\subsection{Proof of Lemma \ref{ortho}}

Using the explicit formula of $\F$ and \eqref{biot unbd}, we write 

\begin{eqnarray*}
 u(x)\cdot \na\F^\e(x) & = &  u^{\perp}(x) \cdot \na^\perp\F^\e(x) \\ 
& = &  -\frac{1}{2\pi\e} \F'\Bigl(\frac{|\Tc(x)|-1}{\e} \Bigl) \int_{\OM}\Bigl(\dfrac{\Tc(x)-\Tc(y)}{|\Tc(x)-\Tc(y)|^2}-\dfrac{\Tc(x)- \Tc(y)^*}{|\Tc(x)- \Tc(y)^*|^2}\Bigl) \om(t,y)\, dy \\
&& \times D\Tc(x)D\Tc^T(x)\frac{\Tc(x)^\perp}{|\Tc(x)|}.
\end{eqnarray*}

As $\Tc$ is  holomorphic, $D\Tc$ is of the form 
$\begin{pmatrix}
a & b \\
-b & a
\end{pmatrix}$
 and we can check that $D\Tc(x)D\Tc^T(x)=(a^2+b^2)Id=|\det(D\Tc)(x)|Id$, so
\begin{equation*}
u(x) \cdot \na\F^\e(x)  =  \frac{\F'(\frac{|\Tc(x)|-1}{\e})|\det(D\Tc)(x)|}{2\pi\e |\Tc(x)|} \int_{\OM}\Bigl(\dfrac{\Tc(y)\cdot \Tc(x)^\perp}{|\Tc(x)-\Tc(y)|^2}-\dfrac{ \Tc(y)^* \cdot \Tc(x)^\perp}{|\Tc(x)- \Tc(y)^*|^2}\Bigl) \om(t,y)\ dy.
\end{equation*}

We compute the $L^1$ norm, next  we change variables twice $\y=\Tc(y)$ and $z=\Tc(x)$, to have

\begin{equation*}
\|u \cdot \na\F^\e\|_{L^1} =\frac{1}{2\pi\e} \int_{|z|\geq 1} \Bigl|\F'\Bigl(\frac{|z|-1}{\e}\Bigl)\Bigl| \Bigl| \displaystyle\int_{|\y|\geq 1} \Bigl(\dfrac{\y \cdot z^\perp/|z|}{|z-\y|^2}-\dfrac{ \y^* \cdot z^\perp/|z|}{|z-\y^*|^2}\Bigl) f(t,\y) \, d\y\Bigl|dz,
\end{equation*}
where $f(t,\y)= \om (t,\Tc^{-1}(\y)) |\det(D\Tc^{-1})(\y)|$.

Thanks to the definition of $\F$, we know that $\Bigl\|\frac{1}{\e}\F'\Bigl(\frac{|z|-1}{\e}\Bigl)\Bigl\|_{L^1}\leq C$. So it is sufficient to prove that
\begin{equation}
\label{tronc}
\Bigl\| \int_{|\y|\geq 1} \Bigl(\dfrac{\y\cdot z^\perp/|z|}{|z-\y|^2}-\dfrac{ \y^*\cdot z^\perp/|z|}{|z- \y^*|^2}\Bigl)f(t,\y)\, d\y\Bigl\|_{L^\infty(1+\e\leq |z|\leq 1+2\e)}\to 0
\end{equation}
as $\e\to 0$, uniformly in time.

Let
\[ A :=\dfrac{\y \cdot z^\perp/|z|}{|z-\y|^2}-\dfrac{ \y^* \cdot z^\perp/|z|}{|z- \y^*|^2}.\]
We compute
\begin{eqnarray*}
A&=& \Bigl(\dfrac{(|z|^2-2 z \cdot \y/|\y|^2+1/|\y|^2)-1/|\y|^2(|z|^2-2z \cdot \y+|\y|^2)}{|z-\y|^2|z- \y^*|^2}\Bigl)\y \cdot \frac{z^\perp}{|z|} \\
&=& \frac{(|z|^2-1)(1-1/|\y|^2)}{|z-\y|^2|z- \y^*|^2}\y \cdot \frac{z^\perp}{|z|}.
\end{eqnarray*}

We now use that $|z|\geq 1$, to write
\[ |z- \y^*|\geq  1-\frac{1}{|\y|}. \]
Moreover, $| \y^*|\leq 1$ allows to have
\[ |z- \y^*|\geq |z|-1. \]

We can now estimate $A$ by:
\[ |A| \leq \frac{(|z|+1)(1+1/|\y|)(|z|-1)^b}{ |z-\y|^2 |z- \y^*|^b} \Bigl| \y \cdot \frac{z^\perp}{|z|}\Bigl| \]
with $0\leq b \leq 1$, to be chosen later. We remark also that $\y \cdot \frac{z^\perp}{|z|}=(\y-z) \cdot \frac{z^\perp}{|z|}$ and the Cauchy-Schwarz inequality gives 
\[ \Bigl|\y \cdot \frac{z^\perp}{|z|}\Bigl|\leq |\y-z|. \]

We now use the fact that $|z|-1 \leq 2\e$, to estimate (\ref{tronc}):
\[ \Bigl| \int_{|\y|\geq 1}A f(t,\y) \, d\y \Bigl|\leq (2+2\e).2.(2\e)^b \int_{|\y|\geq 1} \frac{|f(t,\y)|}{|z-\y||z- \y^*|^b}d\y, \]
hence, the H\"older inequality gives
\[ \Bigl| \int_{|\y|\geq 1}A f(t,\y) \, d\y \Bigl|\leq (2+2\e).2.(2\e)^b  \Bigl\|\frac{|f(t,\y)|^{1/p}}{|z-\y|}\Bigl\|_{L^p}  \Bigl\|\frac{|f(t,\y)|^{1/q}}{|z- \y^*|^b}\Bigl\|_{L^q}\]
with $1/p+1/q=1$ chosen later.

In the same way we estimate $R_2$ in the proof of Proposition \ref{biot est}, we obtain for $bq=1$:
\[ \Bigl\|\frac{|f(t,\y)|^{1/q}}{|z- \y^*|^b}\Bigl\|_{L^q}=\Bigl(\int_{|\y|\geq 1} \frac{|f(t,\y)|}{|z-\y^*|}d\y\Bigl)^{1/q}\leq C_q, \]
where we have used that $\om$ belongs to $L^\infty(\R^+;L^1\cap L^\infty(\OM))$.

Now we use Lemma \ref{ift} for $f\in L^1\cap L^{p_0}$, with $p_0>2$ and for $f\in L^1\cap L^{\infty}$ (see the proof of Proposition \ref{biot est}). Then, we choose $p\in (1,2)$ such that $p_0> \frac{2}{2-p}$ and we follow the estimate of $R_1$ in the proof of Proposition \ref{biot est} to obtain:
\[ \Bigl\|\frac{|f(t,\y)|^{1/p}}{|z-\y|}\Bigl\|_{L^p}=\Bigl(\int_{|\y|\geq 1} \frac{|f(t,\y)|}{|z-\y|^p}d\y\Bigl)^{1/p} \leq C_p. \]
We have used again that $\om$ belongs to $L^\infty(\R^+;L^1\cap L^\infty(\OM))$.

Fixing a $p\in (1,2)$ such that $p_0> \frac{2}{2-p}$, it gives $q\in (2,\infty)$ and $b \in (0,1/2)$ and it follows 
\[ \|u \cdot \na\F^\e\|_{L^1} \leq C(2+2\e).2.(2\e)^{b}C_{p} C_{q} \]
which tends to zero when $\e$ tends to zero, uniformly in time.

\subsection{Proof of Lemma \ref{W11}}

Let $\mathbf{T}>0$ fixed. We rewrite \eqref{biot unbd}:
\begin{eqnarray*} 
u(x) &=& \frac{1}{2\pi} D\Tc^T (x) \Bigl( \int_{\OM} \Bigl(\frac{\Tc(x)-\Tc(y)}{|\Tc(x) - \Tc(y)|^2}- \frac{\Tc(x)-\Tc(y)^*}{|\Tc(x)-\Tc(y)^*|^2}\Bigl)^\perp \om(y)\, dy + \a \frac{\Tc(x)^\perp}{|\Tc(x)|^2}\Bigl)\\
&:= & \frac{1}{2\pi} D\Tc^T(x) h(\Tc(x))
\end{eqnarray*}
where $\a$ is bounded by $\|\g\|_{L^\infty([0,\mathbf{T}])} + \| \om \|_{L^\infty(L^1)}$ in $[0,\mathbf{T}]$ (see \eqref{g bd}).

We start by treating $h$. We change variable $\y=\Tc(y)$, and we obtain
\begin{eqnarray*}
h(z)&=& \int_{B(0,1)^c} \Bigl(\frac{z-\y}{|z-\y|^2}- \frac{z-\y^*}{|z-\y^*|^2}\Bigl)^\perp \om(\Tc^{-1}(\y)) |\det D\Tc^{-1}(\y)| \, d\y + \a \frac{z^\perp}{|z|^2} \\
&=&\int_{B(0,2)^c}\frac{(z-\y)^\perp}{|z-\y|^2} f(t,\y) \, d\y + \int_{B(0,2)\setminus B(0,1)} \frac{(z-\y)^\perp}{|z-\y|^2} f(t,\y) \, d\y - \int_{B(0,2)^c}\frac{(z-\y^*)^\perp}{|z-\y^*|^2} f(t,\y) \, d\y \\
&&- \int_{B(0,2)\setminus B(0,1)}\frac{(z-\y^*)^\perp}{|z-\y^*|^2} f(t,\y) \, d\y + \a \frac{z^\perp}{|z|^2}\\
&:=& h_1(z)+ h_2(z)-h_3(z)-h_4(z)+\a h_5(z),
\end{eqnarray*} 
with $f(t,\y)=\om(t,\Tc^{-1}(\y)) |\det D\Tc^{-1}(\y)|$ belongs to $L^\infty(L^1\cap L^{p_0}(B(0,2)\setminus B(0,1))$ with some $p_0>2$ and to $L^\infty(L^1\cap L^{\infty}(B(0,2)^c)$ (see the proof of Proposition \ref{biot est}). As $|z|= |\Tc(x)|\geq 1$, we are looking for estimates in $B(0,1)^c$. Obviously we have that 
\[ h_5 \text{ belongs to } L^{\infty}(B(0,1)^c) \text{ and } Dh_5 \text{ belongs to } L^{\infty}(B(0,1)^c).\]

Concerning $h_1$, we introduce $f_1:= f \chi_{B(0,2)^c}$ where $\chi_S$ denotes the characteristic function on $S$. Hence
\[ h_1(z)=\int_{\R^2}\frac{(z-\y)^\perp}{|z-\y|^2}f_1(\y) \, d\y \text{ with } f_1\in L^\infty(\R^+;L^1(\R^2)\cap L^{\infty}(\R^2)).\]
We have used the work made in the proof of Proposition \ref{biot est} about the computation of the $L^p$ norm of $f$ in terms of $\om$. The standard estimates on Biot-Savart kernel in $\R^2$ and Calderon-Zygmund inequality
give that
\[ h_1 \text{ belongs to } L^\infty(\R^+\times B(0,1)^c) \text{ and } Dh_1 \text{ belongs to } L^\infty(\R^+; L^{p}(B(0,1)^c)), \ \forall p\in (1,\infty).\]

For $h_2$, is almost the same argument: we introduce $f_2:= f \chi_{B(0,2)\setminus B(0,1)}$, hence
\[ h_2(z)=\int_{\R^2}\frac{(z-\y)^\perp}{|z-\y|^2}f_2(\y) \, d\y \text{ with } f_2\in L^\infty(\R^+;L^1(\R^2)\cap L^{p_0}(\R^2)).\]
The standard estimates on Biot-Savart kernel in $\R^2$ and Calderon-Zygmund inequality
give that
\[ h_2 \text{ belongs to } L^\infty(\R^+\times B(0,1)^c) \text{ and } Dh_2 \text{ belongs to } L^\infty(\R^+;L^{p_0}(B(0,1)^c)).\]

For $h_3$, we can remark that for any $\y \in B(0,2)^c$ we have $|z-\y^*|\geq \frac12$. Therefore, the function $(z,\y)\mapsto \frac{(z-\y^*)^\perp}{|z-\y^*|^2} $ is smooth in $B(0,1)^c\times B(0,2)^c$, which gives us, by a classical integration theorem, that
\[ h_3 \text{ belongs to } L^\infty(\R^+\times B(0,1)^c) \text{ and } Dh_3 \text{ belongs to } L^\infty(\R^+\times B(0,1)^c).\]

To treat the last term, we change variables $\theta = \y^*$
\[ h_4(z) = \int_{B(0,1)\setminus B(0,1/2)} \frac{(z-\theta)^\perp}{|z-\theta|^2} f(\theta^*) \frac{d\theta}{|\theta|^4}:= \int_{\R^2} \frac{(z-\theta)^\perp}{|z-\theta|^2} f_4(\theta)\, d\theta,\]
with $f_4(\theta):=\displaystyle \frac{f(t,\theta^*)}{|\theta|^4} \chi_{B(0,1)\setminus B(0,1/2)}(\theta)$ which belongs to $L^\infty(\R^+;L^1(\R^2)\cap L^{p_0}(\R^2))$. Therefore, standard estimates on Biot-Savart kernel and Calderon-Zygmund inequality
give that
\[ h_4 \text{ belongs to } L^\infty(\R^+\times B(0,1)^c) \text{ and } Dh_4 \text{ belongs to } L^\infty(\R^+;L^{p_0}(B(0,1)^c)).\]

Now, we come back to $u$. As $u (x) = \frac{1}{2\pi} D\Tc^T(x) h(\Tc(x))$, with $D\Tc$ belonging to $L^1_{\loc}(\overline{\OM})$ (see Remark \ref{DT loc}) and $h\circ \Tc$ uniformly bounded, we have that
\[ u \text{ belongs to } L^{\infty}([0,\mathbf{T}];L^1_{\loc} (\overline{\OM})).\]
Adding the bounded behavior of $D\Tc$ at infinity, we have that 
\[  u \text{ belongs to }L^\infty \left([0,\mathbf{T}] ; L^1(\overline{\OM})+ L^\infty(\overline{\OM})\right).\]

Moreover, we have
\begin{eqnarray*}
 |Du(x)| &\leq&  \frac{1}{2\pi}\Bigl( |D^2 \Tc(x)| |h(\Tc(x))| +|D\Tc(x)|^2 |(-Dh_3+\a Dh_5) (\Tc(x))|\Bigl)\\
 &&+ |D\Tc(x)|^2 |(Dh_1+ Dh_2 - Dh_4) (\Tc(x))|.
 \end{eqnarray*}
 For the first right hand side term, we know that $h\circ \Tc$ is uniformly bounded and that $D^2 \Tc$ belongs to $L^p_{\loc}(\overline{\OM})$ for any $p<4/3$ (see Theorem \ref{grisvard}). We see that the second right hand side term belongs to $L^{\infty}([0,\mathbf{T}] ; L^{4/3}_{\loc} (\overline{\OM}))$ because $D\Tc$ belongs to $L^{8/3}_{\loc}(\overline{\OM})$ and $(-Dh_3-\a Dh_5)(\Tc(x))$ belongs to $L^{\infty}([0,\mathbf{T}] \times \overline{\OM}))$.
 
Concerning the third right hand side term, we use that $\Tc$  holomorphic implies that $D\Tc$ is of the form 
$\begin{pmatrix}
a & b \\
-b & a
\end{pmatrix}$. Hence, we get easily that
\[ |D\Tc(x)|_\infty^2 = (\sup (|a|,|b|))^2 \leq a^2+b^2= |\det D\Tc(x)|.\]
Therefore, changing variables, we have for $i=1,2,4$ and $K$ any compact set of $\overline{\OM}$:
\[ \| |D\Tc| |Dh_i\circ \Tc| \|_{L^{2}(K)} \leq \| Dh_i \|_{L^{2}(\tilde K)}\]
with $\tilde K := \Tc(K)$ a compact set (by the continuity of $\Tc$), which is bounded because $2<p_0$. As $DT$ belongs to $L^p(K)$ for any $p<4$, we have by the Holder inequality that $|D\Tc|^2 |Dh_i\circ \Tc|$ is uniformly bounded in $L^p(K)$ for any $p\in [1,4/3)$. Its ends the proof.

\subsection{Proof of Proposition \ref{compact_vorticity}}

We set $\beta(t)=t^2$ and use \eqref{renorm} with this choice. Let $\F \in \mathcal{D}(\R^+\times \RR)$
\begin{equation*}
\int_{\RR} \F(T,x) (\bar\om)^2(T,x)\,dx - \int_{\RR} \F(0,x) (\bar\om)^2(0,x)\,dx =\int_0^T\int_{\RR} (\bar\om)^2 (\dt \F +\bar u\cdot \nabla \F)\,dx \,dt.
\end{equation*}

This is actually an improvement of \eqref{renorm}, in which the equality holds in $L^1_{\loc}(\R^+)$.  Indeed, we have $\pd_t \bar\om=-\diver (\bar u \bar\om)$ (in the sense of distributions) with $\bar\om\in L^\infty$ and $\bar u \in L^\infty_{\loc}(\R^+,L^p_{\loc}(\RR))$ for all $p<4$ (see \eqref{est u}), which implies that $\pd_t \bar\om$ belongs to $L^1_{\loc}(\R^+,W^{-1,p}_{\loc}(\RR))$. Hence, $\bar\om$ belongs to $C(\R^+,W^{-1,p}_{\loc}(\RR))\subset C_{w}(\R^+, L^2_{\loc}(\R^2))$, where $C_{w} L_{\loc}^{2}$ stands for the space of maps $f$ such that for any sequence $t_n\to t$, the sequence $f(t_n)$ converges to $f(t)$ weakly in $L^2_{\loc}$. Since on the other hand $t \mapsto \|\bar \om(t)\|_{L^2}$ is continuous by Remark \ref{remark : conserv}, we have $\bar\om \in C(\R^+,L^2(\RR))$. Therefore the previous integral equality holds for all $T$.

Now, we choose a good test function. We let $\F_0$ be a non-decreasing function on $\R$, which is equal to $1$ for $s\geq 2$ and vanishes for $s\leq 1$ and we set $\F(t,x)=\F_0(|x|/R(t))$, with $R(t)$ a smooth, positive and increasing function to be determined later on, such that $R(0)=R_0$. For this choice of $\F$, we have $(\om_0(x))^2 \F(0,x)\equiv 0$.

We compute then
\begin{equation*}
\na \F= \frac{x}{|x|}\frac{\F_0'}{R(t)}
\end{equation*}
and
\begin{equation*}
\pd_t \F = -\frac{R'(t)}{R^2(t)}|x|\F_0'.
\end{equation*}
We obtain
\begin{eqnarray*}
\int_{\RR} \F(T,x) (\bar\om)^2(T,x)\,dx & =&\int_0^T \int_{\RR} (\bar\om)^2 \frac{\F_0'(\frac{|x|}{R})}{R}\Bigl( \bar u(x) \cdot \frac{x}{|x|}-\frac{R'}{R}|x|\Bigl)\, dx\, dt\\
&\leq& \int_0^T \int_{\RR} (\bar\om)^2 \frac{|\F_0'|(\frac{|x|}{R})}{R} (C -R')\, dx\, dt,
\end{eqnarray*}
where $C$ is independent of $t$ and $x$. Indeed, we have that 
\[ u(t,x)=\dfrac{1}{2\pi}D\Tc^T(x) \Bigl(R[\om](x) +(\g+\int \om_0) \frac{\Tc(x)^\perp}{|\Tc(x)|^2} \Bigl)\]
with  $ |R[\om]| \leq C_1$ (see Proposition \ref{biot est}) and $1/|\Tc(x)|\leq 1$. Using Proposition \ref{T-inf}, we know that there exists a positive $C_2$ such that
\[ |D\Tc({x}) | \leq C_2 |\b|, \ \forall |x|\geq R_0.\]
 Putting together all these inequalities with \eqref{g bd}, we obtain 
\[C=\frac{1}{2\pi}C_2 |\b| \Bigl( C_1 + \|\g\|_{L^\infty([0,\mathbf{T}^*])}+\| \om_0\|_{L^1}\Bigl).\]
Taking $R(t) =R_0 + Ct$, we arrive at
\begin{equation*}
\int_{\RR} \F(T,x) (\bar\om)^2(T,x)\,dx\leq 0,
\end{equation*}
which ends the proof.

\subsection{Proof of Proposition \ref{prop : cont-velocity}} 

By the conservation of the total mass of $\om_i$ \eqref{om-est-1}, we have that
\[ \int_{\R^2} \oti (t,\cdot)  \equiv 0, \ \forall t\geq 0.\]
Moreover, Proposition \ref{compact_vorticity} states that there exists $C_1(\om_0,\OM,\g)$ such that $\om_1(t,\cdot)$ and $\om_2(t,\cdot)$ are compactly supported in $B(0,R_0+C_1 t)$. So we first infer that $\vti (t) \in L^2(\RR)$ for all $t$ (see e.g. \cite{maj-bert}). Using that $\|\om_i\|_{L^1(\OM)\cap L^\infty(\OM)} \in L^\infty(\R^+)$, we even obtain
\begin{equation}
\label{bound-velocity}
\vti \in L_{\loc}^{\infty}(\R^+,L^2(\RR)).
\end{equation}

We now turn to the first assertion in Proposition \ref{prop : cont-velocity}. By Lemma \ref{ift} and Calderon-Zygmund inequality we state that \eqref{om-est-2} implies that  $v_i=K_{\R^2}*\bar \om_i$ belongs to $L^\infty(\R^+\times \RR)$ and its gradient $\na v_i$ to $L^\infty(\R^+,L^4(\RR))$. On the other hand, since the vorticity $\om_i$ is compactly supported, we have for large $|x|$
\begin{equation*}
|v_i(t,x)|\leq \frac{C}{|x|} \int_{\RR} |\bar \om_i(t,y)|\, dy,
\end{equation*}
hence $v_i$ belongs to $L_{\loc}^{\infty}(\R^+,L^p(\RR))$ for all $p>2$. It follows in particular that
\begin{equation*}
v_i \in L_{\loc}^{\infty}(\R^+,W^{1,4}(\RR))
\end{equation*}
and also that $v_i \otimes v_i$ belongs to $L_{\loc}^\infty(L^{4/3})$. Since $v_i$ is divergence-free, we have $v_i \cdot \na v_i=\diver(v_i \otimes v_i)$, and so $v_i \cdot \na v_i \in L_{\loc}^2\big(\R^+,W^{-1,\frac{4}{3}}(\RR)\big)$.

Thanks to \eqref{est u}, we know that $v_i(t)\otimes w_i(t)$ belongs to $L_{\loc}^{4/3}$. At infinity, we use the explicit formula of $u$ \eqref{biot unbd}, the compact support of the vorticity and the behavior of $\Tc$ at infinity (Proposition \ref{T-inf}) to note that $w_i$ is bounded by $C/|x|$. $v_i$ has the same behavior at infinity, which belongs to $L^{8/3}$. This yields
\begin{equation*}
\diver(v_i\otimes w_i),\:\:\diver(w_i\otimes v_i)\in
L_{\loc}^2\big(\R^+,W^{-1,\frac{4}{3}}(\RR)\big).
\end{equation*}
 Besides, we can infer from the behavior of $\Tc$ on the boundary (Theorem \ref{grisvard}) and Proposition \ref{biot est} that $\tilde g_{v_i,\g_0}$, defined in \eqref{g_o_bis}, is uniformly bounded in $L^1(\pd\OM)$. Then we deduce from the embedding of $W^{1,4}(\RR)$ in $C_0^0(\RR)$ that $\tilde g_{v_i,\g_0} \delta_{\OM}$ belongs to $L_{\loc}^2(W^{-1,\frac{4}{3}})$. Therefore, $v_i \tilde g_{v_i,\g_0} \delta_{\OM}\in L_{\loc}^2(\R^+,W^{-1,\frac{4}{3}}(\RR)).$

According to \eqref{vit_equa}, we finally obtain
$$\langle \pd_t  v_i ,\F\rangle =\langle \pd_t v_i -\na p_i,\F\rangle \leq C\|\F\|_{L^2(W^{1,4}_\s)}$$
for all divergence-free smooth vector field  $\F$. This implies that
\begin{equation*}
\pd_t v_i \in L_{\loc}^2\big(\R^+,W^{-1,4/3}_\s(\RR)\big), \quad i=1,2,
\end{equation*}
and the same holds for $\pd_t \vti$. Now, since $\tilde v$ belongs to $L^2_{\loc}\big(\R^+,W^{1,4}_\s\big)$, we deduce from \eqref{bound-velocity} and Lemma 1.2 in Chapter III of \cite{temam} that $\tilde v$ is almost everywhere equal to a function continuous from $\R^+$ into $L^2$ and we have in the sense of distributions on $\R^+$:
\begin{equation*}
 \frac{d}{dt} \|\tilde v\|_{L^2(\RR)}^2 = 2 \langle \pd_t \tilde v,\tilde v \rangle_{W^{-1,4/3}_\s,W^{1,4}_\s}.
\end{equation*}
We finally conclude by using the fact that $\vti(0)=0$.

\section{Final remarks and comments}\label{sect : 6}

\subsection{No extraction in convergence results}

In \cite{taylor,lac_small, GV_lac}, the existence of a weak solution is a consequence of a compactness argument. Indeed, we consider therein the unique solutions $u_n$  of the Euler equations on  the smooth domain $\Omega_n$, which converges to $\OM$ in some senses. Then, in these articles, we extract a subsequence such that $u_{\f(n)} \to u$ and we check that $u$ is solution of the Euler equations in $\OM$. Putting together the present result with \cite{GV_lac}, we can state the following.

\begin{theorem} Let $\om_0$, $\g_0$, $\OM$ as in Theorems \ref{main 1} or \ref{main 2}. For any sequence of  smooth open simply connected domains (or exterior of simply connected domains) $\OM_n$ converging to $\OM$ in the Hausdorff sense, then the unique solution $u_n$ of the Euler equations on $\OM_n$, with initial datum $u_n^0$ such that
\[\diver u^0_n = 0, \  \curl u^0_n =\om_0, \  u^0_n \cdot \hat n\vert_{\pd \Omega_n} = 0, \   \lim_{|x| \rightarrow +\infty} u^0_n  = 0 ,  \  \oint_{\pd \OM_n} u^0_n \cdot \hat \tau\, ds=\g_0 \text{ (only for exterior domains)} , \]
converges in $L^2_{\loc}(\R^+\times \overline{\OM})$ to the unique solution $u$ of the Euler equations on $\OM$ with initial datum $u^0$ such that
\[\diver u^0 = 0, \  \curl u^0 =\om_0, \  u^0 \cdot \hat n\vert_{\pd \Omega} = 0, \   \lim_{|x| \rightarrow +\infty} u^0  = 0 ,   \  \oint_{\pd \OM} u^0 \cdot \hat \tau\, ds=\g_0 \text{ (only for exterior domains)}.  \]
\end{theorem}

\subsection{Special vortex sheet}

In \cite{lac_euler}, we consider some smooth domains $\OM_\e$ which shrink to a $C^2$ Jordan arc $\G$ as $\e$ tends to zero. For $\om_0 \in L^\infty_c(\G^c)$ and $\g\in \R$ given, we denote by $(u_\eps,\omega_\eps)$ the corresponding regular solutions of the Euler equations on $\Pi_\e:= \R^2\setminus \OM_\e$. Up to a truncated smoothly over a size  $\eps$ around the obstacle, it is proved therein that the resulting truncations 
$\tilde u_\eps$ and $\tilde \omega_\eps$, defined over the whole of $\R^2$,  converge in appropriate topologies to   the solutions $\tilde u$, $\tilde \omega$  of the system
\begin{equation} \label{vorticityformulation2}
\left\{
\begin{aligned}
& \pd_t \tilde \omega + \tilde u \cdot \na  \tilde \omega = 0, \quad t > 0, \:  x \in \R^2, \\
& \diver \tilde u = 0,    \quad t > 0, \:  x \in \R^2, \\
&\curl \tilde u  = \tilde \omega +  g_{\tilde \omega,\g} \delta_{\G}, \quad t > 0, \:  x \in \R^2.
\end{aligned}
\right.
\end{equation}
This is an Euler like equation, modified by a Dirac mass along the arc. The density function $g_{\tilde \omega,\g}$ is given explicitly in terms of $\tilde \omega$ and ${\G}$. Moreover, it is shown that it is equal to the jump of the tangential component of the velocity across the arc. We refer to \cite{lac_euler} for all necessary details.

Actually, the presence of this additional measure is mandatory in order that the velocity $\tilde u$ is tangent to the curve, with circulation $\g$ around it.

Therefore, in the exterior of a Jordan arc, \eqref{vorticityformulation2} appears to be a special vortex sheet, ``special'' because the support of the dirac mass does not move (staying to be $\G$) and because the normal component of the velocity on the curve is equal to zero. For a general vortex sheet, we can prove that the normal component is continuous, but not necessarily zero. In both case, we have a jump of the tangential component.

A consequence of the present work is the uniqueness of a solution of  \eqref{vorticityformulation2}, with the good sign conditions for $\om_0$ and $\g$ (see Theorem \ref{main 2}).

For instance, if we assume that $\G$ is the segment $[(-1,0);(1,0)]$, then we have the explicit expression of the harmonic vector field thanks to the Joukowski function, and we can find in \cite[p. 1144]{lac_euler} the following:
$$\curl H_\G = \frac{1}{\pi} \frac{1}{\sqrt{1-x_1^2}} \chi_{(-1,1)}(x_1) \d_0(x_2).$$
Then, choosing $\om_0\equiv 0$ and $\gamma=1$, we have proven that the stationary shear flow $u(t,x)=H_\G(x)$ is the unique solution of the Euler equations with initial vorticity $\frac{1}{\pi} \frac{1}{\sqrt{1-x_1^2}} \chi_{(-1,1)}(x_1) \d_0(x_2)$.

Adding a vorticity or considering other shape for $\G$ complicates a lot the expression of $g_{\tilde \omega,\g}$ (see \cite{lac_euler}). In particular, we do not prove the uniqueness for the so-called Prandtl-Munk vortex sheet: $\frac{1}{\pi} \frac{x_1}{\sqrt{1-x_1^2}} \chi_{(-1,1)}(x_1) \d_0(x_2)$.

\subsection{Extension for constant vorticity near the boundary} As it is remarked several times, the crucial point is to prove that the vorticity never meets the boundary if we consider an initial vorticity compactly supported in $\OM$. However, we can extend easily this result to the case of an initial vorticity constant to the boundary. 
Indeed, for $\a\in \R$ given, choosing $\beta(t)=(t-\a)^2$ in the proof of Proposition \ref{constant_vorticity_2} gives in the same way the following.
\begin{proposition}
Let $\om$ be a global weak solution of \eqref{transport*} such that $\om_0$ is compactly supported in $\overline{\OM}$ and such that $\om_0\equiv \a$ in a neighborhood of the boundary.
If $\om_0$ is non-positive and $\g_0\geq -\int \om_0$ (only for exterior domains), then, for any ${\mathbf{T}^*}>0$, there exists a neighborhood $U_{{\mathbf{T}^*}}$ of $\pd{\OM}$ such that
\begin{eqnarray*}
\om(t)\equiv \a \qquad \textrm{on \; \;} U_{{\mathbf{T}^*}},\qquad \forall t\in [0,{\mathbf{T}^*}].
\end{eqnarray*}
\end{proposition}
Therefore, in the proof of the uniqueness, we still have on $U$
\[ \curl \tilde v = \curl v_1-\curl v_2 = \a-\a=0,\]
which implies that the velocity $\tilde v$ is harmonic near the boundary, allowing us to follow exactly the proof made in Section \ref{sect : 5}.

\subsection{Liapounov and sign conditions}

Let us present in this subsection the different Liapounov functions, the advantage of each, and why it is specific to the case studied.

{\em Vortex wave system in $\R^2$.} Let us consider that the initial vorticity is composed on a regular part plus a dirac mass centered at the point $z(t)$. Then Marchioro and Pulvirenti proved in \cite{mar_pul} that there exists one solution to the following system:
\begin{equation*}\begin{cases}
v(t,\cdot)=(K_{\R^2} \ast \omega)(\cdot,t),\\
 \dot{z}(t)=v(t,z(t)),\\
\dot{\,\phi}_x(t)=v(t,\phi_x(t))+ \frac{(\phi_x(t)-z(t))^\perp}{2\pi |\phi_x(t)-z(t)|^2},\\
\phi_x(0)=x, \; x\neq z_0 ,\\
\omega(t,\phi_x(t))=\omega_0(x),
\end{cases}
\end{equation*}
which means that the point vortex $z(t)$ moves under the velocity field $v$ produced by the regular part $\om$ of the vorticity, whereas the regular part and the vortex point give rise to a smooth flow $\phi$ along which $\om$ is constant. In this case, we can prove that the trajectories never meet the point vortex considering the following Liapounov function:
\[ L(t) := - \ln | \phi_x(t) - z(t) |,\]
for $x\neq z_0$ fixed. We note that $L$ goes to $+\infty$ iff $\phi_x(t) \to z(t)$, so we want to prove that $L$ stays bounded. Next we compute:
\[ L'(t) =- \frac{ ( \phi_x(t) - z(t)) \cdot (\dot{\phi_x}(t) - \dot{z}(t) ) }{ | \phi_x(t) - z(t) |^2} =- \frac{ ( \phi_x(t) - z(t)) \cdot (v(t,\phi_x(t)) - v(t,z(t)) ) }{ | \phi_x(t) - z(t) |^2}.\]
Next, we use that the regular part $v$ is log-lipschitz in order to obtain a Gronwall-type inequality. To summarize, we remark that in the case, the important points are:
\[ L(t) \to \infty \text{ iff } \phi_x(t) \to z(t) \quad \text{and} \quad  ( \phi_x(t) - z(t))  \cdot  \frac{(\phi_x(t)-z(t))^\perp}{2\pi |\phi_x(t)-z(t)|^2}\equiv 0\]
removing the singular part.

\bigskip

{\em Dirac mass fixed in $\R^2$.} Marchioro in \cite{mar} studied exactly the same problem as above, assuming that the vortex mass cannot move. Therefore, the previous Liapounov does not work, because we do not have a difference of two velocities and we cannot use the log-lipschitz regularity. In this article, the author introduced a new Liapounov:
\[ L(t) :=  -\int_{\R^2} \Bigl(\ln|\phi_x(t) -y|\Bigl) \om(t,y)\, dy   - \ln | \phi_x(t) - z_0 |,\]
where the first integral is the stream function associated to $v$. Then, the first step was to prove that this integral is bounded, which implies that $L$ goes to $+\infty$ iff $\phi_x(t) \to z_0$. Next, he computed:
\begin{eqnarray*}
 L'(t) &=&-\Bigl( \int_{\R^2} \frac{\phi_x(t) -y}{|\phi_x(t) -y|^2} \om(t,y)\, dy + \frac{ \phi_x(t) - z_0 }{ | \phi_x(t) - z_0 |^2} \Bigl)\cdot \dot{\phi_x}(t) -\int_{\R^2} \Bigl(\ln|\phi_x(t) -y|\Bigl) \pd_t \om(t,y)\, dy \\
 &=&-\int_{\R^2} \Bigl(\ln|\phi_x(t) -y|\Bigl) \pd_t \om(t,y)\, dy = -\int_{\R^2} \nabla \Bigl(\ln|\phi_x(t) -y|\Bigl) \cdot \Bigl(v(t,y)+ \frac{(y-z_0)^\perp}{2\pi |y-z_0|^2}\Bigl)\om(t,y)\, dy
\end{eqnarray*}
Next, the second step was to prove some good estimate for the right hand side integral in order to conclude by the Gronwall lemma. Here, we see that the singular term is now passed in a integral, which is bounded. Similarly, we note that the important points in this case are:
\[ L(t) \to \infty \text{ iff } \phi_x(t) \to z_0 \quad \text{and} \quad \Bigl( \int_{\R^2} \frac{\phi_x(t) -y}{|\phi_x(t) -y|^2} \om(t,y)\, dy + \frac{ \phi_x(t) - z_0 }{ | \phi_x(t) - z_0 |^2} \Bigl)\cdot \dot{\phi_x}(t)\equiv 0.\]

\bigskip

{\em Interior or exterior of simply connected domains.} In our case, we have again an explicit formula of the velocity by the Biot-Savart law (see \eqref{biot bd} and \eqref{biot unbd}). As the velocity near the boundary blows up, we have to make appear some cancellation as Marchioro did, in order that the singular part goes in an integral. To do that, we introduce the stream function associated to the velocity:
\[L_1(t,x):= \frac1{2\pi} \int_{\OM} \ln\Bigl( \frac{|\Tc(x)-\Tc(y)|}{|\Tc(x)-\Tc(y)^*||\Tc(y)|} \Bigl)\om(y) \, dy+ \frac{\a}{2\pi}\ln |\Tc(x)| \]
with $\a=0$ in the bounded case. However, as this function tends to zero (instead to $\infty$) when $x \to \pd \OM$ (see Lemma \ref{L1 maj}), we add a logarithm:
\[ L(t):=-\ln | L_1(t,\phi_x(t)) |,\]
and the goal is to prove that $L$ stays bounded. Then, we computed in Section \ref{sect : 3}
\[ L'(t) = -\frac{ \pd_t L_1(t,\phi_x(t)) }{| L_1(t,\phi_x(t)) |},\]
and we proved that $\pd_t L_1$  tends to zero as $\phi_x(t) \to \pd \OM$, comparing the rate with $L_1$. Then, we see here that it is important that $\pd_t L_1$ goes to zero where $L_1$ tends to zero. We managed to prove that $\pd_t L_1$ tends to zero near the boundary, and the sign condition allows us to state that the boundary is the only set where $L_1$ vanishes (see Lemma \ref{L1 est}). For instance, in bounded domain (i.e. $\a=0$) we see that a vorticity with different sign can imply that $L_1=0$ somewhere else than on $\pd \OM$. This last remark is the main reason of the sign condition of the vorticity. Next, the sign condition on the circulation follows from the fact that we want the same sign for both terms in $L_1$. 

Therefore, one difference with the case studied by Marchioro is that the stream function of the harmonic vector field does not blow-up. To conclude, let us mention that the Liapounov method is specific to the case studied and it is hard to adapt for other cases. For example, we have presented here the case of dirac mass when $\dot{z}(t)=v(t,z(t))$, when $\dot{z}(t)=0$, but we do not know how to prove for other dynamics, like e.g. $\dot{z}(t)=(1,0)$.

\begin{remark}
In Section \ref{sect : 5}, we have proved the uniqueness up to the time $T$, only using that the vorticity does not meet the boundary between $[0,T]$. Therefore, without any sign condition about the initial vorticity, it is an easy consequence of the uniform estimate of the velocity far away the boundary that we have local uniqueness for any $\om_0\in L^\infty_c (\OM)$. The main part of this paper is to prove the global uniqueness.
 \end{remark}

\section*{Acknowledgment}

I want to thank Milton C. Lopes Filho and Benoit Pausader for early discussions about how to find the good Liapounov function. I also want to thank David G\'erard-Varet and Olivier Glass for several fruitful discussions, and Hungjie Dong for giving me the reference \cite{kenig} with a couter-exemple of a $C^1$ domain which does not verify the elliptic regularity for the laplace problem. Finally, I would like to give many thanks to Isabelle Gallagher and Evelyne Miot for their helpful comments about this report. I'm partially supported by the Agence Nationale de la Recherche, Project MathOc\'ean, grant ANR-08-BLAN-0301-01 and by the Project ``Instabilities in Hydrodynamics'' financed by Paris city hall (program ``Emergences'') and the Fondation Sciences Math\'ematiques de Paris.

\end{document}